\documentclass[12pt,a4paper,article]{amsart}
\usepackage[latin1]{inputenc}
\usepackage{aeguill}
\usepackage{amsfonts}
\usepackage{amssymb}
\usepackage{xspace}
\usepackage{amsmath}
\usepackage{amsthm}
\usepackage{dsfont}
\usepackage{mathabx}
\usepackage[all,cmtip]{xy}

\usepackage{xcolor}
\usepackage[colorlinks=true,linkcolor=blue,pagebackref,breaklinks]{hyperref}
\usepackage[hmargin={1.25in,1.25in},vmargin={1.5in,1.5in}]{geometry}
\usepackage{fancyhdr}

\newtheorem{thm}{Theorem}[section]
\newtheorem{prop}{Proposition}[section]
\newtheorem{df}{Definition}[section]
\newtheorem{lem}{Lemma}[section]
\newtheorem{cor}{Corollary}[section]
\newtheorem{dflem}{Definition-Lemma}[section]
\newtheorem{hyp}{Hypothesis}[section]
\newtheorem{ex}{Example}[section]
\newtheorem{conj}{Conjecture}[section]
\newtheorem{rem}{Remark}[section]

\newenvironment{dem}{\paragraph{Proof}}
{\begin{flushright}$\Box$\end{flushright}}

\newcommand{\N}{\mathbb{N}}
\newcommand{\Z}{\mathbb{Z}}
\newcommand{\Q}{\mathbb{Q}}
\newcommand{\C}{\mathbb{C}}
\newcommand{\R}{\mathbb{R}}

\newcommand{\AK}{\mathbb{A}_{K}}
\newcommand{\AF}{\mathbb{A}_{F}}
\newcommand{\AL}{\mathbb{A}_{L}}
\newcommand{\AFF}{\mathbb{A}_{F'}}
\newcommand{\AFsub}{\mathbb{A}_{F_{0}}}
\newcommand{\AQ}{\mathbb{A}_{\mathbb{Q}}}

\makeatletter 
\@addtoreset{equation}{section}
\makeatother  

\title{Period relations for automorphic induction and applications, I}
\author{Jie LIN}
\date{\today}

\begin{document}
\maketitle

\begin{abstract}
Let $K$ be a quadratic imaginary field. Let $\Pi$ (resp. $\Pi'$) be a regular algebraic cuspidal representation of $GL_{n}(K)$ (resp. $GL_{n-1}(K)$) which is moreover cohomological and conjugate self-dual. In \cite{harris97}, M. Harris has defined automorphic periods of such a representation. These periods are automorphic analogues of motivic periods. In this paper, we show that automorphic periods are functorial in the case where $\Pi$ is a cyclic automorphic induction of a Hecke character $\chi$ over a CM field. More precisely, we prove relations between automorphic periods of $\Pi$ and those of $\chi$. As a corollary, we refine the formula given by H. Grobner and M. Harris of critical values for the Rankin-Selberg $L$-function $L(s,\Pi\times \Pi')$ in terms of automorphic periods. This completes the proof of an automorphic version of Deligne's conjecture in certain cases.
\end{abstract}

\tableofcontents

\section*{Introduction}
Special values of $L$-functions play an important role in Langlands program. Numerous conjectures predict that special values of $L$-functions reflet arithmetic properties of geometric objects. Most of these conjectures are still open and hard to accessible. 

On the other hand, concrete results on special values of $L$-function appear more and more in automorphic settings. For example, In \cite{harris97}, M. Harris constructed complex invariants associated to certain automorphic representations and showed that the special values of automorphic $L$-function could be interpreted in terms of these invariantes. We believe that these invariants are fuctorial for automorphic induction, base change and endoscopic transfer. This article treats the case of automorphic induction of Hecke characters over a CM field.

\bigskip

Let $K$ be a quadratic imaginary field and $F$ be a CM field which is a cyclic extension of $K$ of degree $n$. Let $\chi$ be a regular conjugate self-dual algebraic Hecke character of $F$ satisfying Hypothesis \ref{hypsupercuspidal}. For $\Psi$ any CM type of $F$, a complex invariant $p(\chi,\Psi)$ can be defined and called a $CM$ period. It is defined as ratio of two rational structures on a cohomological space of a Shimura variety associated to $(\chi,\Psi)$.

Put $\Pi=\Pi(\chi)$ the automorphic induction of $\chi$ from $GL_{1}(\AF)$ to $GL_{n}(\AK)$. We assume $n$ is odd at first. In this case, the cuspidal representation $\Pi$ is regular, algebraic, cohomological and conjugate self-dual. For each integer $0\leq s\leq n$, M. Harris has defined $P^{(s)}(\Pi)$ a complex invariant which is called the automorphic period. It is defined as Petersson inner product of a rational element in the coherent cohomology of a Shimura variety associated to $\Pi$.

Our main result is the following: $\cfrac{P^{(s)}(\Pi)}{p(\check{\chi},\Phi_{s,\chi})}$ is an algebraic number where $\Phi_{s,\chi}$ is a CM type of $F$ which depends only on $s$ and $\chi$.

More precisely, we obtain that $\cfrac{P^{(s)}(\Pi)}{D_{F^{+}}^{1/2}\mathcal{G}(\varepsilon_{K})^{-[\frac{n}{2}]}p(\check{\chi},\Phi_{s,\chi})}$ is an element in $E(\chi)^{\times}$ where $D_{F^{+}}$ is the absolute discriminant of $F^{+}$, $\varepsilon_{K}$ is the Artin character of $\AQ^{\times}$ with respect to $K/\Q$, $\mathcal{G}(\varepsilon_{K})$ is the Gauss sum of $\varepsilon_{K}$ and $E(\chi)$ is a number field associated to $\chi$. The above relation is moreover equivariant under action of $G_{K}$. We refer to the subsequent sections for this notion. 


The idea of the proof is simple. Let $\eta$ be an auxiliary Hecke character of $K$. For certain integer $m$, M. Harris has interpreted $L(m,\Pi\otimes \eta)$ in terms of $P^{(s)}(\Pi)$ and some explicit factors for an integer $0\leq s\leq n$. 

On the other hand,  $L(m,\Pi\otimes \eta)=L(m, \chi\otimes \eta\circ N_{L/K})$. Blasius has shown that the latter is product of $p(\chi\otimes \eta\circ N_{L/K}, \Psi)$ with some simple factors and an algebraic number for a CM type $\Psi$ (c.f. Proposition $1.8.1$ of \cite{harrisCMperiod}). From the construction of CM period, we know that CM periods are multiplicative. Therefore, we obtain relations of automorphic periods of $\Pi$ and CM periods of $\chi$ provided that the special value of $L$-function is non zero.

What remains to do is that for every integer $0\leq s\leq n$, find a proper character $\eta$ such that $P^{(s)}(\Pi)$ involves in a non-vanishing special value of $L$-function of $\Pi\otimes\eta$. The existence of $\eta$ follows from careful calculation and the non vanishing of $L$-function follows from results in \cite{harrisunitaryperiod}.

Let us mention that the conditions for $\chi$ is indispensable. Since $P^{(s)}(\Pi)$ is only defined for representations which descend to a unitary group of sign $(n-s,s)$ at infinity, the conjugate self-dual algebraic condition and Hypothesis  \ref{hypsupercuspidal} are needed for applying results of base change. The regular condition allows us to find certain $m$ (which will be called a critical value in the text) for every integer $0\leq s\leq n$.

When $n$ is even, $\Pi$ is no longer algebraic. We can modify $\Pi$ and get similar result. More details can be found in section \ref{evendim}.

\bigskip

These relations of periods lead to a refinement of the formula of $L(m+\frac{1}{2},\Pi\times \Pi')$ by H. Grobner and M. Harris. Here $\Pi'$ is a cuspidal representation of $GL_{n-1}(K)$ which satisfies the same conditions as $\Pi$. In \cite{harrismotivic}, it is proved under regular conditions that $L(m+\frac{1}{2},\Pi\times \Pi')$ equals to some archimedean local factors times automorphic periods up to multiplication by an algebraic number. Therefore, if both $\Pi$ and $\Pi'$ come from Hecke characters, we can now substitute the automorphic periods by CM periods. On the other hand, Blasius's result also gives an interpretation of $L(m+\frac{1}{2},\Pi\times \Pi')$ in terms of CM periods. Combining the two equations, we obtain that the product of these archimedean factors equals to a power of $2\pi i$ up to multiplication by an algebraic number.

Note that the archimedean factors only depend on the infinite part of $\Pi$ and $\Pi'$. This formula should apply to more general $(\Pi,\Pi')$ which do not necessarily come from Hecke characters. This is true for almost all $(\Pi,\Pi')$ in the settings of \cite{harrismotivic}. We assume Hypothesis \ref{m=0} to avoid the only exceptional case where $m=0$. We believe that this Hypothesis is always true but we have not found a proof at this moment.

It follows that $L(m+\frac{1}{2},\Pi\times \Pi')$ equals to a power of $2\pi i$ times automorphic periods up to multiplication by an algebraic number. 

Let us now turn to the setting of motive. Deligne has conjectured in \cite{deligne79} that the critical value of $L$-function of a motive equals to a power of $2\pi i$ times Deligne's period up to multiplication by an algebraic number.  In \cite{harris97}, M. Harris has constructed a motive associated to the representations $\Pi$ and $\Pi'$. In \cite{harrismotivic}, the authors have shown that the product of the automorphic periods appeared in the formula of $L(m+\frac{1}{2},\Pi\times \Pi')$ consist formally with Deligne's period. We then deduce a proof of an automorphic version of Deligne's conjecture in our case.

\bigskip

\paragraph{\textbf{Acknowledgement}}
I would like to express my sincere gratitude to my advisor Michael Harris for suggesting this problem and approach, for reading and correcting earlier versions carefully, and for his exemplary guidance, patience and encouragement. I would also like to thank Harald Grobner for helpful conversations. This article relies highly on their works.

\begin{section}{CM period}
 
\begin{subsection}{Notation}\label{notation}
\text{ }

We fix an embedding $\bar{\Q}\hookrightarrow \C$.


Let $F^{+}$ be a totally real field of degree $n$ over $\Q$ and $F$ be a quadratic imaginary extension of $F^{+}$ and hence a CM field. We denote $\Sigma_{F}$ (resp. $\Sigma_{F^{+}}$) the set of complex embeddings of $F$ (resp. $F^{+}$). Let $\iota\in Gal(F/F^{+})$ be the complex conjugation.

For $z\in \C$, we write $\bar{z}$ for its complex conjugation. For $\sigma\in \Sigma_{F}$, we define $\bar{\sigma}:=\iota\circ \sigma$ the complex conjugation of $\sigma$.

Let $\Phi$ be a subset of $\Sigma_{F}$. We say that $\Phi$ is a CM type if $\Phi\cup \iota \Phi=\Sigma_{F}$ and $\Phi\cap \iota \Phi=\varnothing$. We fix $\{\sigma_{1},\sigma_{2},\cdots,\sigma_{n}\}$ a CM type of $F$. It follows that all the complex embeddings of $F$ are $\sigma_{1}$, $\sigma_{2}$, $\cdots$, $\sigma_{n}$, $\bar{\sigma}_{1}$, $\bar{\sigma}_{2}$, $\cdots$, $\bar{\sigma}_{n}$.

Let $\chi$ be an Hecke character of $F$, i.e. a continuous complex valued character of $F^{\times}\backslash \AF^{\times}$. We assume that $\chi$ is \textbf{algebraic}, namely its infinity type $\chi_{\infty}(z)$ is of the form $\prod\limits_{i=1}^{n}\sigma_{i}(z)^{a_{i}}\bar{\sigma}_{i}(z)^{b_{i}}$ with $a_{i}$, $b_{i}\in \Z$.

We assume moreover:
\begin{itemize}
\item $\chi$ is \textbf{motivic}, i.e. $a_{i}+b_{i}=-w(\chi)$ which is an integer independent of $i$.
\item $\chi$ is \textbf{critical}, i.e. $a_{i}\neq b_{i}$ for all $i$.
\end{itemize}

We can then define $\Phi_{\chi}$, a unique CM type associated to $\chi$, as follows: for each $i$,  $\sigma_{i}\in \Phi_{\chi}$ if $a_{i}<b_{i}$, otherwise $\bar{\sigma}_{i}\in \Phi_{\chi}$. In this case, we say that $\chi$ is \textbf{compatible} with $\Phi_{\chi}$.

Since $\chi$ is algebraic, one can define $E(\chi_{\infty})$, a number field, as in $(1.1.2)$ of \cite{harrisCMperiod}). It is a field of definition of a certain line bundle of a Shimura variety (see the appendix of \cite{harrisappendix} for more details). We denote by $E(\chi)$ the field generated by the values of $\chi$ on $\AF\text{}_{,f}$ over $E(\chi_{\infty})$. It is still a number field. We may assume that $E(\chi)$ contains of $F$.

For every $\sigma\in G_{\Q}:=Gal(\bar{\Q}/\Q)$, we can define $\chi^{\sigma}$ as the composition: $\chi^{\sigma}: \AF \xrightarrow{\chi} E(\chi) \xrightarrow{\sigma} \bar{\Q}$. It is easy to see that $\chi^{\sigma}$ depends only on $\sigma|_{E(\chi)}$ which is in fact an embedding of $E(\chi)$ into $\bar{\Q}$. We denote the set of embeddings of $E(\chi)$ into $\bar{\Q}$ by $\Sigma_{E(\chi)}$ as for the field $F$. For $\sigma\in \Sigma_{E(\chi)}$, we take $\widetilde{\sigma}$ to be any lift of $\sigma$ in $G_{\Q}$. We define $\chi^{\sigma}$ to be $\chi^{\widetilde{\sigma}}$. Obviously, it does not depend on the choice of $\widetilde{\sigma}$. Note that the set $\{\chi^{\sigma},\sigma\in \Sigma_{E(\chi)}\}$ is the same as the set $\{\chi^{\sigma},\sigma\in G_{\Q}\}$.\\

For $s$ a fixed complex number, $(L(s,\chi^{\sigma}))_{\sigma\in \Sigma_{E(\chi)}}$ is an element in $\C^{\Sigma_{E(\chi)}}$. Moreover, when $\sigma$ runs over all the elements in $\Sigma_{E(\chi)}$,  $L(s,\chi^{\sigma})$ takes over all the values in $\{L(s,\chi^{\sigma})|\sigma \in G_\Q\}$. Finally, we identify $\C^{\Sigma_{E(\chi)}}$ with $E(\chi)\otimes \C$ by the inverse of the map which sends $t\otimes z$ to $(\sigma(t)z)_{\sigma\in \Sigma_{E(\chi)}}$ for all $t\in E(\chi)$ and $z\in \C$.

More generally, let $E$ be a number field and $\{ a(\sigma) \} _{\sigma\in G_{\Q}(\text{ resp. }Aut(\C))}$ be some complex numbers which satisfy that

$$a(\sigma)=a(\sigma')\text{ if }  \sigma|_{E}=\sigma'|_{E}\text{ for any }\sigma,\sigma'\in G_{\Q}\text{ (resp. }Aut(\C))  \hspace{3em} (*).$$

We can define $a(\sigma)$ for $\sigma\in \Sigma_{E}$ as previously. Sometimes we simply write $(a)_{\sigma\in \Sigma_{E}}$ instead of $(a(\sigma))_{\sigma\in \Sigma_{E}}$. It is an element of $E\otimes \C$.

\begin{df}
Let $E'$ be a subfield of $\C$, we say $A\sim _{E;E'} B$ if one of the following condition is satisfied:

\begin{enumerate} 
\item $A=0$, 
\item $B=0$ or 
\item both $A,B\in (E\otimes \C)^{\times}$ with $AB^{-1} \in (E\otimes E')^{\times} \subset (E\otimes \C)^{\times}$.
\end{enumerate}

Moreover, we simply note $\sim_{E;\Q}$ by $\sim_{E}$.
\end{df}

\begin{rem}\label{transitive}
Note that this relation is symmetric but not transitive. More precisely, if $A$, $B$, $C\in E\otimes \C$ with $A\sim_{E;E'}B$ and $A\sim_{E;E'}C$, we do not know whether $B\sim_{E;E'} C$ in general unless the condition $A\neq 0$ is provided.

\end{rem}

\begin{rem}
The given identification $E\otimes \C \cong \C^{\Sigma_{E}}$ is a morphism of algebras where the multiplication on the latter is the usual multiplication through each coordinates. This fact implies the following lemma:

\end{rem}

\begin{lem}\label{anotherdef}
Let $E'$ be a subset of a number field $E$. Let $( a(\sigma) ) _{\sigma\in Aut(\C)}$ and $(b(\sigma))_{\sigma\in Aut(\C)}$ be some complex numbers which satisfy $(*)$. Hence $(a(\sigma))_{\sigma\in \Sigma_{E}}$ and $(b(\sigma))_{\sigma\in \Sigma_{E}}$ can be defined.

We assume $b(\sigma)\neq 0$ for all $\sigma$. It is equivalent to saying that $(b(\sigma))_{\sigma\in \Sigma_{E}}\in (E\otimes \C)^{\times}$.

We have $(a(\sigma))_{\sigma\in \Sigma_{E}}\sim_{E;E'} (b(\sigma))_{\sigma\in \Sigma_{E}}$ if and only if

$$\tau\left(\cfrac{a(\sigma)}{b(\sigma)}\right)=\cfrac{a(\tau\sigma)}{b(\tau\sigma)}$$
 for all $\tau\in Aut(\C)$ with $\tau|_{E'}=Id$ and all $\sigma\in Aut(\C)$.

\end{lem}

\begin{rem}

If we have moreover $\cfrac{a(\sigma)}{b(\sigma)}\in \bar{\Q}$ for all $\sigma$, we can replace $Aut(\C)$ by $G_{\Q}$ in the above lemma.
\end{rem}

In this article, we shall consider the action of $G_{K}:=Gal(\overline{\Q}/K)$ but not the action of $G_{\Q}$ where $K$ is a quadratic imaginary field. We fix $K$ a quadratic imaginary field and an embedding $K\hookrightarrow \bar{\Q}$. We write $\iota: K\hookrightarrow \bar{\Q}$ its complex conjugation, i.e. $\iota(z)=\bar{z}$ for all $z\in K$ where $c$ or $\bar{\text{ }}$ stands for complex conjugation.

Let $E$ be a number field containing $K$. We  denote $\Sigma_{E;K}$ the set of embeddings of $E$ into $\C$ which is trivial on $K$. For $( a(\sigma) ) _{\sigma\in G_{K}}$ some complex numbers which satisfy that $a(\sigma)=a(\sigma')$ if $\sigma|_{E}=\sigma'|_{E}$ for any $\sigma,\sigma'\in G_{K}$, we can define $a(\sigma)$ for $\sigma\in \Sigma_{E;K}$ by taking $\tilde{\sigma}$ any lift of $\sigma$ in $G_K$ and defining $a(\sigma)$ to be $a(\tilde{\sigma})$. Let $b(\sigma)_{\sigma \in G_{K}}$ be some complex numbers with the same property. We assume $b(\sigma)\neq 0$ for all $\sigma\in G_{K}$.

\begin{dflem}
We fix $\sigma_{0}\in \Sigma_{E;K}$ and assume $\cfrac{a(\sigma_{0})}{b(\sigma_{0})}\in \overline{\Q}$. 

We have $(a(\sigma))_{\sigma\in \Sigma_{E;K}}\sim_{E;K} (b(\sigma))_{\sigma\in \Sigma_{E;K}}$ if and only if

$$\tau\left(\cfrac{a(\sigma_{0})}{b(\sigma_{0})}\right)=\cfrac{a(\tau\sigma_{0})}{b(\tau\sigma_{0})}$$
 for all $\tau\in G_{K}$.

In this case, we say $a \sim_{E} b$ \textbf{equivariant under action of} $G_{K}$.
In particular, we have $\cfrac{a(\sigma)}{b(\sigma)}\in E$ for all $\sigma\in G_{K}$. 
\end{dflem}

\begin{rem}
If we fix $\sigma_{0}:E\hookrightarrow \overline{\Q}$ and identify $E$ with its image, the condition $\tau\left(\cfrac{a(\sigma_{0})}{b(\sigma_{0})}\right)=\cfrac{a(\tau\sigma_{0})}{b(\tau\sigma_{0})}$ implies that $\cfrac{a(\sigma_{0})}{b(\sigma_{0})}\in E$. Therefore, $a(\sigma_{0})\sim_{E} b(\sigma_{0})$ in the sense that the two complex numbers equal up to multiplication by an element in $E$.

Bearing in mind that the notation $(a(\sigma))_{\sigma\in \Sigma_{E;K}}\sim_{E;K} (b(\sigma))_{\sigma\in \Sigma_{E;K}}$ is the same with that $a \sim_{E} b$ equivariant under action of $G_{K}$. Both notation will be used for reasons of convenience.

\end{rem}
\end{subsection}

\begin{subsection}{CM period}\label{alpha}

Let $\chi$ be a motivic critical character of a CM field $F$. For any $\Psi\subset \Sigma_{F}$ such that $\Psi\cap \iota\Psi=\varnothing$,  one can associate a non zero complex number $p_{F}(\chi,\Psi)$ which is well defined modulo $E(\chi)^{\times}$ (c.f. the appendix of \cite{harrisappendix}). We call it a \textbf{CM period}. Sometimes we write $p(\chi,\Psi)$ instead of $p_{F}(\chi,\Psi)$ if there is no ambiguity concerning the base field $F$.

The CM periods have many good properties. We list below some of them which will be useful in the following sections.

\begin{prop}\label{propCM}

Let $F$ be a quadratic imaginary extension of a totally real field $F^{+}$ as before. Let $F_{0}\subset F$ be a CM field.

Let $\eta$ be a motivic critical Hecke character of $F_{0}$, $\chi$, $\chi_{1}$, $\chi_{2}$ be motivic critical Hecke characters of $F$, $\tau\in \Sigma_{F}$ an embedding of $F$ into $\bar{\Q}$ and $\Psi$ a subset of $\Sigma_{F}$ such that $\Psi\cap \iota\Psi=\varnothing$. We take $\Psi=\Psi_{1}\sqcup\Psi_{2}$ a partition of $\Psi$. We then have:

\begin{equation}\label{charmulti}
p((\chi_{1}\chi_{2})^{\sigma}, \Psi^{\sigma})_{\sigma\in \Sigma_{E(\chi_{1})E(\chi_{2})}} \sim _{E(\chi_{1})E(\chi_{2})} ( p(\chi_{1}^{\sigma}, \Psi^{\sigma})p(\chi_{2}^{\sigma}, \Psi^{\sigma})  ))_{\sigma\in \Sigma_{E(\chi_{1})E(\chi_{2})}}.
\end{equation}

\begin{equation}\label{separateCMtype}
(p(\chi^{\sigma}, \Psi_{1}^{\sigma}\sqcup\Psi_{2}^{\sigma}))_{\sigma\in \Sigma_{E(\chi)}}  \sim _{E(\chi) } ( p(\chi^{\sigma}, \Psi_{1}^{\sigma})p(\chi^{\sigma},\Psi_{2}^{\sigma} ) )_{\sigma\in \Sigma_{E(\chi)}}.
\end{equation}

\begin{equation}
(p(\chi^{\sigma}, \iota\Psi^{\sigma}))_{\sigma\in \Sigma_{E(\chi)}}  \sim _{E(\chi)} (p((\chi^{\iota})^{\sigma}, \Psi^{\sigma}))_{\sigma\in \Sigma_{E(\chi)}}.
\end{equation}

\begin{equation}
(p_{F}((\eta\circ N_{\AF/\AFsub})^{\sigma},\tau^{\sigma}) )_{\sigma\in E(\eta)}\sim _{E(\eta)}  (p_{F_{0}}(\eta^{\sigma}, \tau|_{F_{0}}^{\sigma})  )_{\sigma\in E(\eta)}.
\end{equation}

\end{prop}

\begin{rem}

\begin{enumerate}
\item The first three formulas come from Proposition $1.4$, Corollary $1.5$ and Lemma $1.6$ in \cite{harrisCMperiod}. The last formula is a variation of the Lemma $1.8.3$ in \textit{loc.cit}. The idea was explained in the proof of Proposition $1.4$ in \textit{loc.cit}.
\item The Galois group $G_{\Q}$ acts on $\Sigma_{F}$ by composition. Since we have assumed that $E(\chi)$ contains a normal closure of $F$, we see that if $\sigma\in G_{\Q}$ fixes $E(\chi)$, then it acts trivially on $\Sigma_{F}$. Therefore, we may define an action of $\Sigma_{E(\chi)}$ on $\Sigma_{F}$.
\end{enumerate}
\end{rem}

Before proving the above proposition, let us introduce the definition for period associated to special Shimura datum.

Let $(T,h)$ be a Shimura datum where $T$ is a torus defined over $\Q$ and $h:Res_{\C/\R}\mathbb{G}_{m,\C}\rightarrow G_{\R}$ a homomorphism satisfying the axioms defining a Shimura variety. Such pair is called a \textbf{special} Shimura datum. Let $Sh(T,h)$ be the associated Shimura variety and $E(T,h)$ be its reflex field.

Let $(\gamma,V_{\gamma})$ be a one-dimensional algebraic representation of $T$ (the representation $\gamma$ is denoted by $\chi$ in \cite{harrisappendix}). We denote by $E(\gamma)$ a definition field for $\gamma$. We may assume that $E(\gamma)$ contains $E(T,h)$. Suppose that $\gamma$ is motivic (see \textit{loc.cit} for the notion). We know that $\gamma$ gives an automorphic line bundle $[V_{\gamma}]$ over $Sh(T,h)$ defined over $E(\gamma)$. Therefore, the complex vector space $H^{0}(Sh(T,h),[V_{\gamma}])$ has an $E(\gamma)$-rational structure (for the definition of a rational structure, see section \ref{automorphicperiod}), denoted by $M_{DR}(\gamma)$ and called the De Rham rational structure.

On the other hand, the canonical local system $V_{\gamma}^{\triangledown}\subset [V_{\gamma}]$ gives another $E(\gamma)$-rational structure $M_{B}(\gamma)$ on $H^{0}(Sh(T,h),[V_{\gamma}])$, called the Betti rational structure.

We now consider $\chi$ an algebraic Hecke character of $T(\AQ)$ with infinity type $\gamma^{-1}$ (our character $\chi$ corresponds to the character $\omega^{-1}$ in \textit{loc.cit}). Let $E(\chi)$ be the number field generated by the values of $\chi$ on $T(\AQ \text{}_{,f})$ over $E(\gamma)$. We know $\chi$ generates a one-dimensional complex subspace of $H^{0}(Sh(T,h),[V_{\gamma}])$ which inherits two $E(\chi)$-rational structure, one from $M_{DR}(\gamma)$, the other from $M_{B}(\gamma)$. Put $p(\chi,(T,h))$ the ratio of these two rational structures which is well defined modulo $E(\chi)^{\times}$.

Suppose that we have two tori $T$ and $T'$ both endowed with a Shimura datum $(T,h)$ and $(T',h')$. Let $u:(T',h')\rightarrow (T,h)$ be a map between the Shimura data. Let $\chi$ be an algebraic Hecke character of $T(\AQ)$. We put $\chi':=\chi\circ u$ an algebraic Hecke character of $T'(\AQ)$. Since both the Betti structure and the De Rham structure commute with the pullback map on cohomology, we have the following proposition:

\begin{prop}\label{propgeneral}
Let $\chi$, $(T,h)$ and $\chi'$, $(T',h')$ be as above. We have:
$$(p(\chi,(T,h)))_{\sigma\in \Sigma_{E(\chi)}}\sim_{E(\chi)} (p(\chi',(T',h')))_{\sigma \in \Sigma_{E(\chi)}}$$
\end{prop}

\begin{rem}
In Proposition $1.4$ of \cite{harrisCMperiod}, the relation is up to $E(\chi);E(T,h)$ where $E(T,h)$ is a number field associated to $(T,h)$. Here we consider the action of $G_{\Q}$ and can thus obtain a relation up to $E(\chi)$ (see the paragraph after Proposition $1.8.1$ of \textit{loc.cit}).
\end{rem}

For $F$ a CM field and $\Psi$ a subset of $\Sigma_{F}$ such that $\Psi\cap \iota\Psi=\varnothing$, we can define a Shimura datum $(T_{F},h_{\Psi})$ where $T_{F}:=Res_{F/\Q}\mathbb{G}_{m,F}$ is a torus and $h_{\Psi}:Res_{\C/\R}\mathbb{G}_{m,\C} \rightarrow T_{F,\R}$ is a homomorphism such that over $\sigma\in \Sigma_{F}$, the Hodge structure induced on $F$ by $h_{\Psi}$ is of type $(-1,0)$ if $\sigma\in \Psi$, of type $(0,-1)$ if $\sigma\in \iota\Psi$, and of type $(0,0)$ otherwise. By definition, $p_{F}(\chi,\Psi)=p(\chi,(T_{F},h_{\Psi}))$.
We can now prove Proposition \ref{propCM}:
\begin{proof}
We keep the above notation. All the equations in Proposition \ref{propCM} come from Proposition \ref{propgeneral} by certain maps between Shimura data as follows:
\begin{enumerate}
\item The diagonal map $ (T_{F},h_{\Psi})\rightarrow (T_{F}\times T_{F},h_{\Psi}\times h_{\Psi})$ pulls $(\chi_{1},\chi_{2})$ back to $\chi_{1}\chi_{2}$.
\item The multiplication map $T_{F} \times T_{F} \rightarrow T_{F}$ sends $h_{\Psi_{1}}$, $h_{\Psi_{2}}$ to $h_{\Psi_{1}\sqcup \Psi_{2}}$.
\item The complex conjugation $H_{F}\rightarrow H_{F}$ sends $h_{\Psi}$ to $h_{\iota \Psi}$.
\item The norm map $(T_{F},h_{\{\tau\}})\rightarrow (T_{F_{0}},h_{\{\tau|_{F_{0}}\}})$ pulls $\eta$ back to $\eta\circ N_{\AF/\mathbb{A}_{F,0}}$.
\end{enumerate}
\end{proof}


\bigskip
The special values of an $L$-function for a Hecke character over a CM field can be interpreted in terms of CM periods. The following theorem is proved by Blasius. We state it as in Proposition $1.8.1$ in \cite{harrisCMperiod} where $\omega$ should be replaced by $\check{\omega}:=\omega^{-1,c}$ (for this erratum, see the notation and conventions part on page $82$ in the introduction of \cite{harris97}),

\begin{thm}\label{blasius}
Let $F$ be a CM field and $F^{+}$ be its maximal totally real subfield. Put $n$ the degree of $F^{+}$ over $\Q$.

Let $\chi$ be a motivic critical algebraic Hecke character of $F$ and $\Phi_{\chi}$ be the unique CM type of $F$ which is compatible with $\chi$.

For $m$ a critical value of $\chi$ in the sense of Deligne (c.f. Lemma \ref{critical}), we have
$$L(\chi,m) \sim_{E(\chi)} D_{F^{+}}^{1/2}(2\pi i)^{mn}p(\check{\chi},\Phi_{\chi})$$
equivariant under action of $G_{\Q}$ where $D_{F^{+}}$ is the absolute discriminant of $F^{+}$.
\end{thm}

\begin{rem}\label{criticalforcharacter}
\text{}
\begin{enumerate}
\item Let $\{\sigma_{1},\sigma_{2},\cdots,\sigma_{n}\}$ be any CM type of $F$. Let $(\sigma_{i}^{a_{i}}\overline{\sigma}_{i}^{-w-a_{i}})_{1\leq i\leq n}$ denote the infinity type of $\chi$ with $w=w(\chi)$. We may assume $a_{1}\geq a_{2}\geq \cdots \geq a_{n}$. We define $a_{0}:=+\infty$ and $a_{n+1}:=-\infty$ and define $k:=max\{0\leq i\leq n\mid a_{i}>-\cfrac{w}{2}\}$. An integer $m$ is critical for $\chi$ if and only if

\begin{equation}\label{criticalvalue1}
       max(-a_{k}+1,w+1+a_{k+1})\leq m \leq min(w+a_{k},-a_{k+1})
\end{equation}
     (c.f. Lemma \ref{critical}).
\item $D_{F^{+}}^{1/2}$ is well defined up to multiplication by $\pm 1$. More generally, if $\{z_{1},z_{2},\cdots,z_{n}\}$ is any $\Q$-base of $L$, then $det(\sigma_{i}(z_{j}))_{1\leq i,j\leq n}\sim_{\Q} D_{F^{+}}^{1/2}$.

 \end{enumerate}
\end{rem}

\begin{lem}\label{lemmadiscriminant}
 If $F^{+}$ and $F'^{+}$  are two number fields with degree $n$ and $n'$ such that $F^{+}\cap F'^{+}=\Q$, then $D_{F^{+}F'^{+}}^{1/2}\sim_{\Q} (D_{F^{+}}^{1/2})^{n'}(D_{F'^{+}}^{1/2})^{n}$.
\end{lem}

\begin{dem}
 Let$\sigma_{1},\cdots,\sigma_{n}$ denote all the embeddings of $F^{+}$ into $\C$ and $\sigma_{1}',\cdots, \sigma_{n'}'$ all the embeddings of $F'^{+}$ into $\C$. For $1\leq i\leq n$ and $1\leq j\leq n'$, define $\tau_{ij}: F^{+}F'^{+} \rightarrow \C$ with $\tau_{ij}|_{F^{+}}=\sigma_{i}$ and $\tau_{ij}|_{F'^{+}}=\sigma_{i}'$. We take $\{z_{1},z_{2},\cdots,z_{n}\}$ any $\Q$-base of $F$ and $\{w_{1},w_{2},\cdots,w_{n'}\}$ any $\Q$-base of $F'^{+}$. We know $\{z_{i}w_{j}\}_{1\leq i\leq n,1\leq j\leq n'}$ is a $\Q$-base of $F^{+}F'^{+}$. By the remark above, we have:
$$\begin{array}{lcl}
D_{F^{+}F'^{+}}^{1/2} &\sim_{\Q}& det(\tau_{kl}(z_{i}w_{j}))_{1\leq i\leq n,1\leq j\leq n';1\leq k\leq n,1\leq l\leq n'}\\
&\sim_{\Q}& det[\sigma_{k}(z_{i})\sigma'_{l}(w_{j})]_{1\leq i\leq n,1\leq j\leq n';1\leq k\leq n,1\leq l\leq n'}\\
&\sim_{\Q}& det[(\sigma_{k}(z_{i}))_{1\leq i,k\leq n}\otimes (\sigma'_{l}(z_{j}))_{1\leq j,l\leq n'}]\\
&\sim_{\Q}& det[(\sigma_{k}(z_{i}))_{1\leq i,k\leq n}\otimes I_{n'})\times ((\sigma'_{l}(z_{j}))_{1\leq j,l\leq n'}\otimes I_{n}]\\
&\sim_{\Q}& [det(\sigma_{k}(z_{i}))_{1\leq i,k\leq n}]^{n'}\times [det(\sigma'_{l}(z_{j}))_{1\leq j,l\leq n'}]^{n}\\
&\sim_{\Q}& (D_{F^{+}}^{1/2})^{n'}(D_{F'^{+}}^{1/2})^{n}.

\end{array}$$

\end{dem}

\bigskip

In this article, we consider the CM fields which contain the quadratic imaginary field $K$. 

Let $F^{+}$ be a totally real field of degree $n$ over $\Q$. Put $F:=F^{+}K$ a CM field. In the following, we denote by $\sigma_{1},\cdots, \sigma_{n}$ the complex embeddings of $F$ which are trivial on $K$. We then have $\Sigma_{F}=\{\sigma_{1},\sigma_{2},\cdots,\sigma_{n}\}\cup \{\overline{\sigma}_{1},\overline{\sigma}_{2},\cdots,\overline{\sigma}_{n}\}$. In particular, $\{\sigma_{1},\sigma_{2},\cdots,\sigma_{n}\}$ is a CM type of $F$.

We will need the following auxiliary algebraic Hecke character later:
\begin{lem}\label{psi}
There exists $\psi$ an algebraic Hecke character of $K$ with infinity type $z^{1}\overline{z}^{0}$ such that $\psi\psi^{c}=||\cdot||_{\AK}$.
\end{lem}

This is a special case of Lemma $4.1.4$ in \cite{CHT}. We fix one such $\psi$ throughout this paper.

At last, we introduce certain notation and simple lemmas concerning Hecke characters of the quadratic imaginary field $K$.

\begin{df}
For $\eta$ an algebraic Hecke character of $K$ with infinity type $\eta_{\infty}(z)=z^{a(\eta)}\bar{z}^{b(\eta)}$, we define:

 \begin{itemize}

 \item $\check\eta=\eta^{-1,c}$ a Hecke character of $K$.
 \item $\widetilde{\eta}(z)=\eta(z)/\eta(\bar{z})$ a Hecke character of $K$.
 \item $\eta_{0}$ the Hecke character of $\Q$ such that $\eta \eta ^{c}= (\eta_{0}\circ N_{\AK/\AQ})  ||\cdot||^{a(\eta)+b(\eta)}$.
 \item $\eta^{(2)}=\eta^2/\eta_{0}\circ N_{\AK/\AQ}$.

 \end{itemize}
\end{df}

\begin{rem}
\begin{enumerate}
\item $\eta_{0}$ is a character of finite order and thus a Dirichlet character.
\item The two transform $\check{\text{}}$ and $\text{}^{(2)}$ commute, i.e.  $(\check{\eta})^{(2)}=\check{\eta^{(2)}}$. We simply denote it by $\check{\eta}^{(2)}$.
\item It is easy to verify that $\widetilde{\eta}/\check{\eta}^{(2)}=||\cdot||^{a(\eta)+b(\eta)}$.
\item $\widetilde{\eta}$ is conjugate self-dual, i.e. $\widetilde{\eta}^c=\widetilde{\eta}^{-1}$.
\end{enumerate}
\end{rem}

Let $\alpha$, $\beta$ be Hecke characters of $K$ with  $\alpha_{\infty}(z)=z^{\kappa}$ and $\beta_{\infty}(z)=z^{-k}$, $\kappa$,  $k\in \Z$. For $\varepsilon$ a Dirichlet character of $\Q$ of conductor $f$, we define the Gauss sum of $\varepsilon$ by $\mathcal{G}(\varepsilon):=\sum\limits_{a=1}^{f}\varepsilon(a)e^{2\pi i a/f}$.

Then by $(1.10.9)$ and $(1.10.10)$ on page $100$ of \cite{harris97}, we have
\begin{lem}\label{lemmas}
\begin{equation}\label{dirichlet}
(\mathcal{G}(\alpha_{0}))_{\sigma\in \Sigma_{E(\alpha)}}\sim _{E(\alpha)} (p((\alpha_{0}\circ N_{\AK/\AQ}),1)^{-1})_{\sigma\in \Sigma_{E(\alpha)}};
\end{equation}

\begin{equation}\label{norm}
p(||\cdot||_{\AK}^{\kappa},1)\sim_{\Q} (2\pi i)^{-\kappa}.
\end{equation}

\end{lem}

We can then deduce easily from the previous lemma and Proposition \ref{propCM} that:

\begin{lem}
\begin{eqnarray}
((2\pi i)^{\kappa} \mathcal{G}(\alpha_{0}))_{\sigma\in \Sigma_{E(\alpha)}}
& \sim_{E(\alpha)}&( [p(||\cdot||^{\kappa},1)p(\alpha_{0}\circ N_{\AK/\AQ},1)]^{-1})_{\sigma\in \Sigma_{E(\alpha)}}\nonumber\\
&\sim_{E(\alpha)} &(p(\alpha\alpha^{c},1)^{-1})_{\sigma\in \Sigma_{E(\alpha)}}.
\end{eqnarray}

\begin{eqnarray}
((2\pi i)^{k}p(\check{\beta}^{(2)},1))_{\sigma\in \Sigma_{E(\beta)}}&\sim_{E(\beta)}& (p(||\cdot||^{-k},1)p(\check{\beta}^{(2)},1))_{\sigma\in \Sigma_{E(\beta)}}\nonumber\\
&\sim_{E(\beta)} &(p(\widetilde{\beta},1))_{\sigma\in \Sigma_{E(\beta)}}.
\end{eqnarray}

\end{lem}

\end{subsection}
\end{section}

\begin{section}{Base change}
\begin{subsection}{General base change}

Let $G$ and $G'$ be two connected quasi-split reductive algebraic groups over $\Q$. Put $\widehat{G}$ the complex dual group of $G$. The Galois group $G_{\Q}=Gal(\overline{\Q}/\Q)$ acts on $\widehat{G}$. We define the \textbf{$L$-group} of $G$ by $\text{}^{L}G:=\widehat{G} \rtimes G_{\Q}$ and we define $\text{}^{L}G'$ similarly. A group homomorphism between two $L$-groups $\text{}^{L}G\rightarrow \text{}^{L}G'$ is called an \textbf{$L$-morphism} if it is continuous, its restriction to $\widehat{G}$ is analytic and it is compatible with the projections of $\text{}^{L}G$ and $\text{}^{L}G'$ to $G_{\Q}$. If there exists an $L$-morphism $\text{}^{L}G\rightarrow \text{}^{L}G'$, the \textbf{Langlands' principal of functoriality} predicts a correspondence from automorphic representations of $G(\AQ)$ to automorphic representations of $G'(\AQ)$ (c.f. section $26$ of \cite{arthurtraceformula}). More precisely, we wish to associate an $L$-packet of automorphic representations of  $G(\AQ)$ to that of $G'(\AQ)$.

Locally, we can specify this correspondence for unramified representations. Let $v$ be a finite place of $\Q$  such that $G$ is unramified at $v$. We fix $\Gamma_{v}$ a maximal compact hyperspecial subgroup of $G_{v}:=G(\Q_{v})$.  By definition, for $\pi_{v}$ an admissible representation of $G_{v}$, we say $\pi_{v}$ is \textbf{unramified} (with respect to $\Gamma_{v}$) if it is irreducible and $dim \pi_{v}^{\Gamma_{v}} > 0$. One can show that $\pi_{v}^{\Gamma_{v}}$ is actually one dimensional since $\pi_{v}$ is irreducible.

Denote $H_{v}:=\mathcal{H}(G_{v},\Gamma_{v})$ the Hecke algebra consisting of compactly supported continuous functions from $G_{v}$ to $\C$ which are $\Gamma_{v}$ invariants at both side. We know $H_{v}$ acts on $\pi_{v}$ and preserves $\pi_{v}^{\Gamma_{v}}$ (c.f. \cite{minguez}). Since $\pi_{v}^{\Gamma_{v}}$ is one-dimensional, every element in $H_{v}$ acts as a multiplication by a scalar on it. In other words, $\pi_{v}$ thus determines a character of $H_{v}$. This gives a map from the set of unramified representations of $G_{v}$ to the set of characters of $H_{v}$ which is in fact a bijection (c.f. section $2.6$ of \cite{minguez}).

We can moreover describe the structure of $H_{v}$ in a simpler way. Let $T_{v}$ be a maximal torus of $G_{v}$ contained in a Borel subgroup of $G_{v}$. We denote $X_{*}(T_{v})$ the set of cocharacters of $T_{v}$ defined over $\Q_{v}$. The Satake transform identifies the Hecke algebra $H_{v}$ with the polynomial ring $\C[X_{*}(T_{v})]^{W_{v}}$ where $W_{v}$ is the Weyl group of $G_{v}$ (c.f. section $1.2.4$ of \cite{harristakagi}).

Let $G'$ be a connected quasi-split reductive group which is also unramified at $v$. We can define $\Gamma'_{v}$, $H'_{v}:=\mathcal{H}(G'_{v},\Gamma'_{v})$ and $T'_{v}$ similarly. An $L$-morphism $\text{}^{L}G\rightarrow \text{}^{L}G'$ induces a morphism $\widehat{T_{v}}\rightarrow \widehat{T'_{v}}$ and hence a map $T'_{v}\rightarrow T_{v}$. The latter gives a morphism from $\C[X_{*}(T'_{v})]^{W'_{v}}$ to $\C[X_{*}(T_{v})]^{W_{v}}$ and thus a morphism from $H'_{v}$ to $H_{v}$. Its dual hence gives a map from the set of unramified representations of $G_{v}$ to that of $G'_{v}$. This is the local Langlands's principal of functoriality for unramified representations.

\bigskip

In this article, we are interested in a particular case of the Langlands' functoriality: the cyclic base change. Let $K/\Q$ be a cyclic extension, for example $K$ is a quadratic imaginary field. Let $G$ be a connected quasi-split reductive group over $\Q$. Put $G'=Res_{K/\Q}G_{K}$. We know $\widehat{G'}$ equals to $\widehat{G}^{[K:\Q]}$. The diagonal embedding is then a natural $L$-morphism $\text{}^{L}G \rightarrow \text{}^{L}G'$. The corresponding functoriality is called the base change.

More precisely, let $v$ be a place of $\Q$. The above $L$-morphism gives a map from the set of unramified representations of $G(\Q_{v})$ to that of $G'(\Q_{v})$ where the latter is isomorphic to $G(K_{v})\cong \bigotimes\limits_{w|v} G(K_{w})$. By the tensor product theorem, all the unramified representation of $G(K_{v})$ can be written uniquely as tensor product of unramified representations of $G(K_{w})$ where $w$ runs over the places of $K$ above $v$. We fix $w$ a place of $K$ above $v$ and we then get a map from the set of unramified representation of $G(\Q_{v})$ to that of $G(K_{w})$. For an unramified representation of $G(\Q_{v})$, we call the image its \textbf{base change} with respect to $K_{w}/\Q_{v}$.

Let $\pi$ be an admissible irreducible representation of $G(\AQ)$. We say $\Pi$, a representation of $G(\AK)$, is a \textbf{(weak) base change} of $\pi$ if for all $v$, a finite place of $\Q$, such that $\pi$ is unramified at $v$ and all $w$, a place of $K$ over $v$, $\Pi_{w}$ is the base change of $\pi_{v}$. In this case, we say $\Pi$ \textbf{descends to $\pi$} by base change. Moreover, $\Pi$ and $\pi$ have the same partial $L$-function.

\begin{rem}
The group $Aut(K)$ acts on $G(\AK)$. This induces an action of $Aut(K)$ on automorphic representations of $G(\AK)$, denoted by $\Pi^{\sigma}$ for $\sigma\in Aut(K)$ and $\Pi$ an automorphic representation of $G(\AK)$. It is easy to see that if $\Pi$ is a base change of $\pi$, then $\Pi^{\sigma}$ is one as well. So if we have the strong multiplicity one theorem for $G(\AK)$, we can conclude that every representation in the image of base change is $Aut(K)$-stable.

\end{rem}

\begin{ex}\textbf{Base change for $GL_{1}$}

Now let us give an example of base change. Let $L/K$ be a cyclic extension of numbers fields. Let $\sigma$ be a generator of $Gal(L/K)$. The automorphic representations of $GL_{1}$ over a number field are nothing but Hecke characters. Let $\eta$ be a Hecke character of $K$. The base change of $\eta$ to $GL_{1}(\AL)$ in this case is just $\eta\circ N_{\AL/\AK}$.

On the other hand, as discussed above, if $\chi$ is a Hecke character of $\AL$, then a necessary condition for it to descend is $\Pi =\Pi^{\sigma}$ for all $\sigma\in Gal(L/K)$. We shall see that this is also sufficient.
\begin{thm}\label{bccharacter}
Let  $L/K$ be a cyclic extension of number fields and $\chi$ be a Hecke character of $L$. If $\chi=\chi^{\sigma}$ for all $\sigma\in Gal(L/K)$, then there exists $\eta$, a Hecke character of $K$, such that $\chi=\eta\circ N_{\AL/\AK}$.
Moreover, if $\chi$ is unramified at some place $v$ of $L$, we can choose $\eta$ to be unramified at places of $L$ over $v$.
\end{thm}
\begin{proof}
We define at first $\eta_{0}:N_{\AL/\AK}(\AL)^{\times}\rightarrow \C$ as follows:
for $w\in N_{\AL/\AK}(\AL)^{\times} $, take $z\in \AL^{\times}$ such that $w=N_{\AL/\AK}(z)$ and we define $\eta_{0}(w)=\chi(z)$. This does not depend on the choice of $z$. In fact, if $z'\in\AL^{\times}$ such that $N_{\AL/\AK}(z')$ also equals $w=N_{\AL/\AK}(z)$, then by Hilbert's theorem $90$ which says $H^{1}(Gal(L/K), \AL^{\times})=1$, there exists $t\in \AL^{\times}$ such that $z'=\cfrac{\sigma(t)}{t}z$ for some $\sigma\in Gal(L/K)$. Hence $\chi(z')=\cfrac{\chi^{\sigma}(t)}{\chi(t)}\chi(z)=\chi(z)$. Therefore $\eta_{0}(w)$ is well defined.
One can verify that $\eta_{0}$ is a continuous character.

By Hasse norm theorem, $N_{\AL/\AK}(\AL)^{\times}\cap K^{\times} =N_{\AL/\AK}(L)^{\times}$ (this is a direct corollary of Hilbert's theorem $90$ on $K^{\times}\backslash\AK^{\times}$ that $H^{1}(Gal(L/K), K^{\times}\backslash\AK^{\times})=\{1\}$), we know $\eta_{0}$ is trivial on $N_{\AL/\AK}(\AL)^{\times}\cap K^{\times}$, and hence factors through $(K^{\times}\cap N_{\AL/\AK}(\AL)^{\times})\backslash\N_{\AL/\AK}(\AL)^{\times}$. The latter is a finite index open subgroup of $K^{\times}\backslash \AK^{\times}$ by the class field theory. We can thus extend $\eta_{0}$ to a Hecke character of $\eta$ as we want.
\end{proof}

\end{ex}

\end{subsection}

\begin{subsection}{Base change for unitary groups}

In the following sections, $K$ is always a quadratic imaginary field.

Let $r,s\in \N$ such that $r+s=n$. Fix $q=vv^{c}$ a place of $\Q$ which splits in $K$ such that $v$, $v^{c}$ are inert in $F$.

Take $D_{r,s}$ to be a division algebra of dimension $n^{2}$ with center $K$ endowed with $*: D_{r,s}\rightarrow D_{r,s}$ an involution of second kind. Moreover, we want $D_{r,s}$ to be unramified at all finite places which does not divide $q$ and ramified at places over $q$. We want also the unitary group associated to $(D_{r,s},*)$ has infinity sign $(r,s)$. The calculation of local invariants of unitary groups in chapter $2$ of \cite{clozelIHES} shows that such a division algebra exists.

We denote $G=U(r,s)$ the restriction from $K$ to $\Q$ of the unitary group associate to $(D_{r,s},*)$. Let $\pi$ be a cuspidal automorphic representation of $U(r,s)(\AQ)$.

\begin{rem}
The Unitary group $G$ is not quasi-split in general. Globally we do not know how to define the base change map. But $G$ is quasi-split at almost every finite places. Therefore we can still define the weak base change map.
\end{rem}
Note that $G(\AK)\cong GL_{n}(\AK)$. We have the strong multiplicity one theorem for $G(\AK)$. Therefore, if $\pi$ admits a base change then it is unique. Moreover, as discussed in the previous section, its base change is invariant under the action of $Aut(K)$ if it exists. If we identify $G(\AK)$ with $GL_{n}(\AK)$, the action of the non trivial element in $Aut(K)$ acts on the latter by sending $g$ to $\overline{g}^{-1}$. Therefore, a representation of $GL_{n}(\AK)$ is $Aut(K)$-stable if and only if it is conjugate self-dual.

For the existence of base change under conditions, we refer to \cite{harrislabesse} Theorem $2.2.2$ and Theorem $3.1.3$.

\begin{df}\label{cohomological}
For $G$ a reductive group over $\Q$, we denote by $\mathfrak{g}_{\infty}$ the Lie group of $G(\R)$ and $K_{\infty}$ a maximal compact subgroup of $G(\R)$. We say $\pi$, a representation of $G(\AQ)$ is \textbf{cohomological} if there exists $W$, an algebraic finite-dimensional representation of $G(\R)$ such that $H^{*}(\mathfrak{g}_{\infty},K_{\infty}; \pi\otimes W)\neq 0$. In this case, we say $\pi$ is cohomological for $W$.
\end{df}

We restate the existence of base change in the cohomological case here( c.f. Theorem $3.1.3$ of \cite{harrislabesse}):

\begin{thm}\label{bcunitaryexistence}
Let $\pi$ be a cuspidal automorphic representation of $U(r,s)(\AQ)$. Assume $\pi$ is cohomological. If $(\pi_q)$ is supercuspidal, then the weak base change of $\pi$ exists and is unique. Moreover, it is conjugate self-dual as a representation of $GL_{n}(\AK)$ and cohomological.
\end{thm}

\begin{rem}
The original theorem in \cite{harrislabesse} requires that the Jacquet-Langlands transfer $JL(\pi_{q})$ to $GL_{n}(\Q_{q})$ is supercuspidal. Recall here $q$ is a split in $K$. We have $U(r,s)(\Q_{q})\cong GL_{n}(\Q_{q})$. 
\end{rem}

Let $\Pi$ be a cuspidal representation of $GL_{n}(\AK)$. We can consider it as a representation of $G(\AK)$. We are more interested in the ''going down'' part here, i.e., to decide when $\Pi$, a cuspidal representation of $G(\AK)\cong GL_{n}(\AK)$, descends to a representation of $G(\AQ)$. A necessary condition is that $\Pi$ is $Aut(K)$-stable, i.e., it is conjugate self-dual in the sense of representation of $GL_{n}(\AK)$. In contrast to the case of $GL_{1}$, this is not sufficient. The non sufficiency can be already seen for $U(1)$:

\begin{ex}\textbf{Base change for $U(1)$}

Let $K$ be a quadratic imaginary field. Let $U$ be the one dimensional torus over $\Q$ defined by $U(\Q)=ker(N: K^{\times}\rightarrow \Q^{\times})$ where $N$ is the norm map. We have $U_{K}\cong GL_{1}|_{K}$.

Let $\eta$ be a continuous character of $U(\Q)\backslash U(\AQ)$. The base change of $\eta$ with respect to $K/\Q$ is the character of $\AK^{\times}$ who sends $z\in \AK$ to $\eta(\cfrac{z}{\overline{z}})$.

On the other hand, let $\chi$ be a Hecke character of $\AK$. We want to know when $\chi$ descends to a character of $U(\Q)\backslash U(\AQ)$, i.e. $\chi(z)=\eta(\cfrac{z}{\overline{z}})$ for some $\eta$. As discussed above, a necessary condition is that $\chi =\chi^{c,-1}$.

Hilbert's theorem $90$ shows that the morphism $K^{\times}\backslash \AK^{\times}\rightarrow U(\Q)\backslash U(\AQ)$ which sends $z$ to $\cfrac{z}{\overline{z}}$ is surjective. We then get an isomorphism
$$\begin{array}{ccc}
K^{\times}\AQ^{\times}\backslash \AK^{\times}& \xrightarrow{\sim}& U(\Q)\backslash U(\AQ)\\
z&\mapsto &\cfrac{z}{\overline{z}}
\end{array}$$
We can then deduce that:
\begin{prop}
A Hecke character $\chi$ of $\AK^{\times}$  descends to a Hecke character of $U(\AQ)$ if and only if it is trivial on $\AQ^{\times}$.
\end{prop}

 Let $\chi$ be a Hecke character of $K$ such that $\chi=\chi^{c,-1}$. This condition is equivalent to saying that $\chi$ is trivial on $N(\AK^{\times})$. Thus $\chi|_{\AQ^{\times}}$ factors through $\Q^{\times}N(\AK^{\times})\backslash\AQ^{\times}$. By class field theory, the latter is isomorphic to $Gal(K/\Q)\cong \Z/2\Z$. Therefore either $\chi$ itself is trivial on $\Q^{\times}N(\AK^{\times})\backslash\AQ^{\times}$ or $\chi\otimes \varepsilon_{K}$ is trivial on $\Q^{\times}N(\AK^{\times})\backslash\AQ^{\times} $ where $\varepsilon_{K}$ is the Artin character of $\AQ^{\times}$ with respect to $K/\Q$. We take $\widetilde{\varepsilon_{K}}$ to be a lift of ${\varepsilon_{K}}$ to $K^{\times}\backslash \AK^{\times}$. We may assume that $\widetilde{\varepsilon_{K}}$ is conjugate self-dual since $\varepsilon_{K}$ is. We then have:
\begin{cor}
Let $\chi$ be a Hecke character of $K$ such that $\chi=\chi^{c,-1}$. Exactly one of $\{\chi,\chi\otimes \widetilde{\varepsilon_{K}}\}$ is a base change of a Hecke character of $U(\AQ)$.
\end{cor}

\begin{rem}
 Roughly speaking, half of the conjugate self-dual Hecke characters of $\AK$ descend to $U(\AQ)$. We have a similar result for general unitary groups, see Theorem $2.4.1$ of \cite{harrislabesse}. This corresponds to the fact that the unitary group occurs twice as the twisted endoscopic group for $GL_{n}$.
\end{rem}

Now we have a simple criteria to check whether $\chi$, a conjugate self-dual character of $\AK^{\times}$, descends to $U(\AQ)$ or not. Let $t=(t_{v})_{v\leq \infty}\in \AQ$ with $t_{v}=1$ for all $v< \infty$ and $t_{\infty}=-1$. Note that $t$ generates $\Q^{\times}N(\AK^{\times})\backslash\AQ^{\times}$. Therefore $\chi$ descends if and only if $\chi(t)=1$. In particular, if $\chi$ is algebraic, then the infinity type of $\chi$ is of the form $(\cfrac{z}{\overline{z}})^{a}$ with $a\in \Z$. Hence $\chi(t)=1$ and $\chi$ descends.
\begin{cor}
If $\chi$ is an algebraic conjugate self-dual character of $\AK^{\times}$, then $\chi$ descends to a character of $U(\AQ)$.
\end{cor}

\end{ex}
\bigskip

For the general cases, as for $U(1)$, if $\Pi$ is in addition cohomological, we have the following result for ''going down'':

\begin{thm}\label{bcunitary}
Let $\Pi$ be a cuspidal cohomological  conjugate self-dual representation of $GL_{n}(\AK)$. If $\Pi$ is supercuspidal at the places over $q$, then $\Pi$ is the weak base change of a cuspidal representation $\pi$ of $U(r,s)(\AQ)$. Moreover, $\pi$ is cohomological.
\end{thm}

\begin{rem}
This follows from Theorem $2.1.2$ and Theorem $3.1.2$ of \cite{harrislabesse}. Theorem $3.1.2$ in the \textit{loc.cit} describes the "going down" part for the unique quasi-split unitary group. Theorem $2.1.2$ gives us a way to associate representations of two different unitary groups.
\end{rem}

\end{subsection}

\begin{subsection}{Base change for similitude unitary groups}
We write $GU(r,s)$ for the rational similitude group of $U(r,s)$, i.e. for $R$ any $\Q$-algebra,$$GU(r,s)(R)=\{g\in GL(D_{r,s} \otimes_{\Q}R)\mid  gg^{*}=\lambda(g)Id, \lambda(g) \in R^{\times}\}.$$

We have an exact sequence of $\Q$-groups:

$$1\rightarrow U(r,s)\rightarrow GU(r,s) \rightarrow \mathbb{G}_{m} \rightarrow 1.$$

This exact sequence split in $K$. Indeed, by Galois descent, it is enough to define $\theta_{r,s}$, a Galois automorphism on $U(r,s)_{K}\times \mathbb{G}_{m,K}$ such that the subgroup of $U(r,s)_{K}\times \mathbb{G}_{m,K}$ fixed by $\theta_{r,s}$ is isomorphic to $GU(r,s)$. We now define $\theta_{r,s}$ as follows:

For $R$ a $\Q$-algebra, note that $(U(r,s)_{K}\times \mathbb{G}_{m,K})(R)\cong GL(D_{r,s}\otimes_{\Q}R)\times (K\otimes_{\Q}R)$. We define
$$\theta_{r,s}: GL(D_{r,s}\otimes_{\Q}R)\times (K\otimes_{\Q}R)\rightarrow GL(D_{r,s}\otimes_{\Q}R)\times (K\otimes_{\Q}R)$$ by sending $(g,z)$ to $((g^{*})^{-1}\bar{z},\bar{z})$. It is easy to verify that $\theta_{r,s}$ satisfies the condition mentioned above.

We then have that $GU(r,s)_{K}\cong U(r,s)_{K}\times \mathbb{G}_{m,K}$. In particular, $GU(\AK)\cong GL_{n}(\AK)\times \AK^{\times}$. For $\Pi$ a cuspidal representation of $GL_{n}(\AK)$ and $\xi$ a Hecke character of $K$, $\Pi\otimes \xi$ defines a cuspidal representation of $GU(\AK)$. Conversely, by the tensor product theorem, every irreducible automorphic representation of $GU(\AK)$ can be written uniquely up to isomorphisms in the form $\Pi\otimes \xi$.

We give at first a criterion for $\Pi\otimes \xi$ to be $\theta_{r,s}$-stable which can be deduced without any difficulty:

 \begin{lem}
 The representation $\Pi\otimes \xi$ of $GU(r,s)$ is $\theta_{r,s}$-stable if and only if $\Pi$ is conjugate self-dual and $\omega_{\Pi}=\cfrac{\xi^{c}}{\xi}$ where $\omega_{\Pi}$ is the central character of $\Pi$.
 \end{lem}

Let us consider now the base change for $G=GU(r,s)$. Let $\pi$ be a cuspidal representation of $G(\AQ)$. If $\pi$ admits a base change, then it is unique by the strong multiplicity one theorem for $GL_{n}\times GL_{1}$. The following theorem on the existence can be proved as for Theorem $VI.2.1$ in \cite{harristaylor} where we use Theorem \ref{bcunitaryexistence} in stead of Clozel's theorem (Theorem $A.5.2$ of \cite{clozellabesse}) on the existence of base change for $U(n-1,1)$.

\begin{thm}
Let $\pi$ be a cuspidal automorphic representation of $GU(r,s)(\AQ)$. Assume $\pi$ is cohomological. If the Jacquet-Langlands transfer $JL(\pi|_{U(r,s),q})$ to $GL_{n}(\Q_{q})$ is supercuspidal, then the weak base change of $\pi$ exists and then unique. Moreover, it is $\theta_{r,s}$-stable and cohomological.

If we write the base change of $\pi$ in the form $\Pi\otimes \xi$ where $\Pi$ is cuspidal representation of $GL_{n}(\AK)$ and $\xi$ a Hecke character of $K$, then the last statement is equivalent to say that:
\begin{itemize}
\item $\Pi$ is conjugate self-dual and cohomological.
\item $\omega_{\Pi}(z)=\xi(\cfrac{\overline{z}}{z})$ for all $z\in \AK^{\times}$.
\end{itemize}
\end{thm}

For the going down part, we have:

\begin{thm}\label{bcsimilitude}
Let $\Pi\otimes\xi$ be a cuspidal representation of $GU(r,s)(\AK)$ where $\Pi$ is a cuspidal representation of $GL_{n}(\AK)$ and $\xi$ is a Hecke character of $\chi$. Let $W$ be a finite dimensional representation of $GU(r,s)(\R)$ and we denote $W_{K}=W\otimes W$ a representation of $GU(r,s)_{K}(\R)$. Assume these data satisfy:
\begin{itemize}
\item $\Pi$ is conjugate self-dual;
\item $\Pi$ is cohomological for $W_{K}$.
\item $\Pi$ is supercuspidal at places over $q$.
\item $\omega_{\Pi}=\cfrac{\xi^{c}}{\xi}$.
\item $W|_{K_{\infty}^{\times}}=\check{\xi}_{\infty}$ where $K_{\infty}$ is the completion of $K$ at the infinite place and $K_{\infty}^{\times}$ embeds into $GU(r,s)(\R)$.

\end{itemize}

Then $\Pi\otimes \xi$ descends to a cuspidal representation of $GU(r,s)(\AQ)$ which is cohomological for $W$.

\end{thm}
This follows from Theorem \ref{bcunitary}. The proof is the same as Theorem $VI.2.9$ of \cite{harristaylor}.

Sometimes we start with a representation of $GL_{n}(\AK)$. In this case, the following lemma will be useful (c .f. Lemma $VI.2.10$ of \cite{harristaylor}):
\begin{lem}\label{lemmaxi}
Let $\Pi$ be a conjugate self-dual cuspidal representation of $GL_{n}(\AK)$. We assume $\Pi$ is cohomological and supercuspidal at places over $q$. We then have that there exists $\xi$, a Hecke character of $K$ and $W$, a finite dimensional representation of $GU(r,s)(\R)$ such that all the conditions in Theorem \ref{bcsimilitude} are satisfied.
\end{lem}
 \end{subsection}

\end{section}

\begin{section}{Automorphic period}
\begin{subsection}{Motive}
In this article, a \textbf{motive} simply means a pure motive for absolute Hodge cycles in the sense of Deligne \cite{deligne79}. More precisely, a motive over $\Q$ with coefficient in a number field $E$ is given by its Betti realization $M_{B}$, its de Rham realization $M_{DR}$ and its $l$-adic realization $M_{l}$ for all prime numbers $l$ where $M_{B}$ and $M_{DR}$ are finite dimensional vector space over $E$, $M_{l}$ is a finite dimensional vector space over $E_{l}:=E\otimes_{\Q} \Q_{l}$ endowed with:
\begin{itemize}
\item $I_{\infty}: M_{B}\otimes \C \xrightarrow{\sim}  M_{DR}\otimes \C $ as $E\otimes_{\Q}\C$-module;
\item $I_{l}: M_{B}\otimes \Q_{l}\xrightarrow{\sim} M_{l}$ as $E\otimes _{\Q} \Q_{l}$-module.

\end{itemize}
From the isomorphisms above, we see that $dim_{E}M_{B}=dim_{E}M_{DR}=dim_{E_{l}}M_{l}$ and is called the \textbf{rank} of $M$. We need moreover:
\begin{enumerate}
\item An $E$-linear involution (infinite Frobenius) $F_{\infty}$ on $M_{B}$ and a Hodge decomposition $M_{B}\otimes \C=\bigoplus\limits_{p,q\in \Z}M^{p,q}$ as $E\otimes \C$-module such that $F_{\infty}$ sends $M^{p,q}$ to $M^{q,p}$.

For $w$ an integer, we say $M$ is \textbf{pure of weight} $w$ if $M^{p,q}=0$ for $p+q\neq w$. Throughout this paper, all the motives are assumed to be pure. We assume also $F_{\infty}$ acts on $M^{p,p}$ as a scalar for all $p\in \Z$.

We say $M$ is \textbf{regular} if $dim M^{p,q}\leq 1$ for all $p,q\in\Z$.


\item An $E$-rational Hodge filtration on $M_{DR}$: $\cdots \supset M^{i}\supset M^{i+1}\supset \cdots$ which is compatible with the Hodge structure on $M_{B}$ via $I_{\infty}$, i.e.,
$$I_{\infty}(\bigoplus\limits_{p\geq i}M^{p,q})=M^{i}\otimes \C.$$
\item A Galois action of $G_{\Q}$ on each $M_{l}$ such that $(M_{l})_{l}$ forms a compatible system of $l$-adic representations $\rho_{l}:G_{\Q} \longrightarrow GL(M_{l})$. More precisely, for each prime number $p$, let $I_{p}$ be the inertia subgroup of a decomposition group at $p$ and $F_{p}$ the geometric Frobenius of this decomposition group. We have that for all $l\neq p$, the polynomial  $det(1-F_{p}|M_{l}^{I_{p}})$ has coefficients in $E$ and is independent of the choice of $l$.
We can then define $L_{p}(s,M):=det(1-p^{-s}F_{p}|M_{l}^{I_{p}})^{-1}\in E(p^{-s})$ for whatever $l\neq p$.
\end{enumerate}

For any fixed embedding $\sigma: E\hookrightarrow \C$, we may consider $L_{p}(s,M,\sigma)$ as a complex valued function. We define $L(s,M,\sigma)=\prod\limits_{p}L_{p}(s,M,\sigma)$. It converges for $Re(s)$ sufficiently large. It is conjectured that the $L$-function has analytic continuation and functional equation on the whole complex plane.

We can also define $L_{\infty}(s,M)$, the infinite part of the $L$-function, as in chapter $5$ of \cite{deligne79}.

Deligne has defined the critical values for $M$ as follows:

\begin{df}
We say an integer $m$ is \textbf{critical} for $M$ if neither $L_{\infty}(M,s)$ nor $L_{\infty}(\check{M},1-s)$ has a pole at $s=m$ where $\check{M}$ is the dual of $M$. We call $m$ a \textbf{critical value} of $M$.
\end{df}

\begin{rem}\label{infinitytype}
The notion $L_{\infty}(s,M)$ implicitly indicates that the infinity type of the $L$-function does not depend on the choice of $\sigma:E\hookrightarrow \C$. More precisely, for every $\sigma:E\hookrightarrow \C$, put $M_{B,\sigma}:=M_{B}\otimes_{E,\sigma}\C$. We then have $M_{B}\otimes \C=\bigoplus\limits_{\sigma:E\hookrightarrow \C}M_{B,\sigma}$. Since $M^{p,q}$ is stable by $E$, each $M_{B,\sigma}$ inherits a Hodge decomposition $M_{B,\sigma}=\bigoplus M_{B,\sigma}^{p,q}$. We may define $L_{\infty}(s,M,\sigma)$ with help of the Hodge decomposition of $M_{B}\otimes_{E,\sigma}\C$. It is a product of $\Gamma$ factors which depend only on $\dim M_{B,\sigma}^{p,q}$ and the action of $F_{\infty}$ on $M_{B,\sigma}^{p,p}$. The latter is independent of $\sigma$ since we have assumed that $F_{\infty}$ acts on $M^{p,p}$ by a scalar.

It remains to show that $\dim M_{B,\sigma}^{p,q}$ is also independent of $\sigma$. In fact, since $M$ is pure, $M^{p,q}$ can be reconstructed from the Hodge filtration $M^{i}$. Hence $M^{p,q}=\bigoplus M_{B,\sigma}^{p,q}$ is a free $E\otimes\C$-module. One can show $M_{B,\sigma}^{p,q}=M^{p,q}\otimes_{E,\sigma}\C$ and hence $dim M_{B,\sigma}^{p,q}$ is independent of $\sigma$.

\end{rem}

Deligne has defined $c^{+}(M), c^{-}(M)\in (E\times \C)^{\times}$ explicitly from the infinity type of $M$, called the Deligne's period. He has conjectured that:
\begin{conj} If $0$ is critical for $M$, then
$$(L(0,M,\sigma))_{\sigma\in \Sigma_{E}}\sim_{E} c^{+}(M).$$
\end{conj}

More generally, tensoring $M$ by the Tate motive $\Q(m)$ (c.f. \cite{deligne79} chapter $1$), we obtained a new motive $M(m)$. We remark that $L(s,M(m),\sigma)=L(s+m,M,\sigma)$. The following conjecture is a corollary of the previous conjecture:
\begin{conj}
If $m$ is critical for $M$, then
$$(L(m,M,\sigma))_{\sigma\in \Sigma_{E}}\sim_{E} (2\pi i)^{d^{+}n}c^{+}(M) \text{ if } m \text{ is even };$$
$$(L(m,M,\sigma))_{\sigma\in \Sigma_{E}}\sim_{E} (2\pi i)^{d^{-}n} c^{-}(M) \text{ if } m \text{ is odd }$$

where $d^{+}$ and $d^{-}$ are two integers.
\end{conj}

Deligne has given a criteria to determine whether $0$ is critical for $M$ (see $(1.3.1)$ of \cite{deligne79}). We observe that $n$ is critical for $M$ if and only if $0$ is critical for $M(n)$. Thus we can rewrite the criteria of Deligne for arbitrary $n$. In the case where $M^{p,p}=0$ for all $p$, this criteria becomes rather simple.

We first define the \textbf{Hodge type} of $M$ by the set $T=T(M)$ consisting of pairs $(p,q)$ such that $M^{p,q}\neq 0$. Since $M$ is pure, there exists an integer $w$ such that $p+q=w$ for all $(p,q)\in T(M)$. We remark that if $(p,q)$ is an element of $T(M)$, then $(q,p)$ is also contained $T(M)$.

 \begin{lem}\label{critical}
Let $M$ be a pure motive of weight $w$. We assume that for all $(p,q)\in T(M)$, $p\neq q$ which is equivalent to that $p\neq \cfrac{w}{2}$.

Let $p_{1}<p_{2}<\cdots<p_{n}$ be some integers such that $$T(M)=\{(p_{1},q_{1}),(p_{2},q_{2}),\cdots, (p_{n},q_{n})\}\cup \{(q_{1},p_{1}),(q_{2},p_{2}),\cdots, (q_{n},p_{n})\}$$ where $q_{i}=w-p_{i}$ for all $1\leq i\leq n$.

We set $p_{0}=-\infty$ and $p_{n+1}=+\infty$. Denote by $k:=max\{0\leq i \leq n\mid  p_{i}< \cfrac{w}{2} \}$. We have that $m$ is critical for $M$ if and only if
$$\max(p_{k}+1,w+1-p_{k+1})\leq m\leq  \min(w-p_{k},p_{k+1}).$$
In particular, critical value always exist in the case where $p_{i}\neq q_{i}$ for all $i$.
 \end{lem}

\begin{dem}

The Hodge type of $M(m)$ is $\{(p_{i}-m,w-p_{i}-m)\mid 1\leq i\leq n\}\cup \{(w-p_{i}-m,p_{i}-m)\mid 1\leq i\leq n\}$. By Deligne's criteria, $0$ is critical for $M$ if and only if for all $i$, either $p_{i}-m\leq -1$ and $w-p_{i}-m\geq 0$, or $p_{i}-m\geq 0$ and $w-p_{i}-m\leq -1$. Hence the set of critical values for $M$ are  $\bigcap\limits_{1\leq i\leq n} ([w+1-p_{i}, p_{i}]\cup [p_{i}+1,w-p_{i}])$.

For $i\leq k$, $p_{i}<\cfrac{w}{2}$ and then $p_{i}<w+1-p_{i}$. Therefore $\bigcap\limits_{1\leq i\leq k} ([w+1-p_{i}, p_{i}]\cup [p_{i}+1,w-p_{i}])=\bigcap\limits_{1\leq i\leq k}  [p_{i}+1,w-p_{i}]=[p_{k}+1,w-p_{k}]$. The same we have $\bigcap\limits_{k< i\leq n} ([w+1-p_{i}, p_{i}]\cup [p_{i}+1,w-p_{i}])=\bigcap\limits_{k< i\leq n} [w+1-p_{i}, p_{i}]=[w+1-p_{k+1},p_{k+1}]$.

We deduce at last the set of critical values for $M$ is $[\max(p_{k}+1,w+1-p_{k+1}), \min(w-p_{k},p_{k+1})]$. It is easy to verify the latter set is non empty.
\end{dem}

\begin{df}
Let $M$ be a pure motive of weight $w$. We say $M$ is \textbf{polarized} if there exists a morphism of motive $$<,>:M\otimes M\rightarrow \Q(-w)$$
which is a non-degenerate bilinear form in any realization.

We assume this bilinear form is symmetric if $w$ is even and skew-symmetric if $w$ is odd.
\end{df}

\end{subsection}
\begin{subsection}{Motives and periods associated to automorphic representations}\label{automorphicperiod}
\text{}

Let $n\geq 1$ be an integer, $K$ be a quadratic imaginary field and $\Pi=\Pi_{f}\otimes \Pi_{\infty}$ be a regular cohomological cuspidal representation of $GL_{n}(\AK)$. We say a representation is \textbf{regular} if its infinity type $(z^{a_{i}}\overline{z}^{b_{i}})_{1\leq i\leq n}$ satisfies $a_{i}\neq a_{j}$ for all $1\leq i<j\leq n$. By Lemma \ref{wPi} below, this implies $b_{i}\neq b_{j}$ for all $1\leq i<j\leq n$. We recall that $\Pi$ is \textbf{cohomological} if there exists $W$ an algebraic finite-dimensional representation of $GL_{n}(\AK)$ such that $H^{*}(\mathcal{G}_{\infty},K_{\infty}; \Pi\otimes W)\neq 0$ where $\mathcal{G}_{\infty}=Lie(GL_{n})(\C)$ and $K_{\infty}$ is the compact unitary group of $GL_{n}(\C)$. In particular, $\Pi$ is \textbf{algebraic}, i.e. $a_{i}$, $b_{i}\in \Z+\frac{n-1}{2}$ for all $i$ (for more details on the infinity type we refer to \cite{clozelaa} section $3.3$).

We denote $V$ the representation space for $\Pi_{f}$. For $\sigma\in Aut(\C)$, we define another $GL_{n}(\AK\text{}_{f})$-representation $\Pi_{f}^{\sigma}$ to be $V\otimes_{\C,\sigma}\C$. Let $\Q(\Pi)$ be the subfield of $\C$ fixed by $\{\sigma\in Aut(\C) \mid  \Pi_{f}^{\sigma} \cong \Pi_{f}\}$. We call it the \textbf{rationality field} of $\Pi$.

For $E$ a number field, $G$ a group and $V$ a $G$-representation over $\C$, we say $V$ has a \textbf{$E$-rational structure} if there exists a $E$-vector space $V_{E}$ endowed with an action of $G$ such that $V=V_{E}\otimes_{E}\C$ as a representation of $G$. We call $V_{E}$ a $E$-rational structure of $V$. We have the following result (c.f. \cite{clozelaa} Theorem $3.13$):

\begin{thm}
For $\Pi$ a regular algebraic cuspidal representation of $GL_{n}(\AK)$, $\Q(\Pi)$ is a number field. Moreover, $\Pi_{f}$ has a $\Q(\Pi)$-rational structure. For all $\sigma\in Aut(\C)$, $\Pi_{f}^{\sigma}$ is the finite part of a cuspidal representation of $GL_{n}(\AK)$ which is unique by the strong multiplicity one theorem, denoted by $\Pi^{\sigma}$.
\end{thm}
\begin{rem}
\begin{enumerate}
\item This theorem is still true if $K$ is replaced by arbitrary number field (see \textit{loc.cit}).
\item For any $\sigma\in Aut(\C)$, $\Pi^{\sigma}$ is determined by $\sigma|_{\Q(\Pi)}: \Q(\Pi)\hookrightarrow \C$. Therefore, we may define $\Pi^{\sigma}$ for any $\sigma\in \Sigma_{Q(\Pi)}$ by lifting $\sigma$ to an element in $Aut(\C)$. In particular, we may define $\Pi^{\sigma}$ for any $\sigma\in Gal(\overline{\Q}/\Q)$ or $\sigma\in \Sigma_{E}$ where $E$ is an extension of $\Q(\Pi)$.
\end{enumerate}
\end{rem}

It is conjectured that such a $\Pi$ is attached to a motive with coefficients in a finite extension of $\Q(\Pi)$:
\begin{conj}\label{conjmotive}(Conjecture $4.5$ and paragraph $4.3.3$ of \cite{clozelaa})

Let $\Pi$ be a regular algebraic cuspidal representation of $GL_{n}(\AK)$ and $\Q(\Pi)$ its rationality field. There exists $E$ a finite extension of $\Q(\Pi)$ and $M$ a regular motive of rank $2n$ over $\Q$ with coefficients in $E$ such that
$$L(s,M,\sigma)=L(s+\cfrac{1-n}{2},\Pi^{\sigma})$$
for all $\sigma: E\hookrightarrow \C$.

Moreover, if the infinity type of $\Pi$ is $(z^{a_{i}}\overline{z}^{b_{i}})_{1\leq i\leq n}$, then the Hodge type of $M$ is $\{(-a_{i}+\cfrac{n-1}{2},-b_{i}+\cfrac{n-1}{2})\mid 1\leq i\leq n\}\cup \{(-b_{i}+\cfrac{n-1}{2},-a_{i}+\cfrac{n-1}{2})\mid 1\leq i \leq n\}$.
\end{conj}

\begin{lem}(Lemma $4.9$ of \cite{clozelaa})\label{wPi}

Let $\Pi$ be a cuspidal algebraic representation of $GL_{n}(\AK)$. We have that $\Pi$ is \textbf{pure} which means there exists an integer $w_{\Pi}$ such that $(z^{a_{i}}\overline{z}^{b_{i}})_{1\leq i\leq n}$, the infinity type of $\Pi$, satisfies $a_{i}+b_{i}=-w_{\Pi}$ for all $i$.
\end{lem}

\begin{cor}
If the motive associated to $\Pi$, a cuspidal algebraic regular representation of $GL_{n}(\AK)$, exists as in Conjecture \ref{conjmotive}, then it is pure of weight $w_{\Pi}+n-1$.
\end{cor}

 \bigskip

In the following, let $\Pi$ be a regular algebraic cuspidal representation of $GL_{n}(\AK)$ which is moreover cohomological and conjugate self-dual. M. Harris has shown that there exits a motive associated to $\Pi$ (c.f. \cite{harris97}).  Moreover, he defined automorphic periods which are analogues of Deligne's periods. He also proved  that special values of $L$-function for $\Pi$ can be interpreted in terms of these automorphic periods (c.f. Theorem \ref{maintheorem}). We now introduce his results.

In this article, we always assume $\Pi$ is supercuspidal at places over $q$. This is not necessary to define the motive associated to $\Pi$ but will be useful to define the period $P^{(s)}$ for all $s$. We shall see it later (c.f. remark \ref{supercuspidal}).

In order to simplify the notation in Theorem \ref{maintheorem} (see Equation (\ref{mainformula})), we consider $\Pi^{\vee}$, the contragredient representation of $\Pi$. It is regular, cuspidal, cohomological, conjugate self-dual and supercuspidal at places over $q$. By Lemma \ref{lemmaxi}, for each $(r,s)$, we can find $\xi_{r,s}$, a Hecke character of $K$ such that $\Pi^{\vee}\otimes \xi_{r,s}$ satisfies the conditions in Theorem \ref{bcsimilitude}. Therefore, $\Pi^{\vee}\otimes \xi_{r,s}$ descends to a cuspidal cohomological representation of $GU(r,s)(\AQ)$, denoted by $\pi$. We write $W(\pi_{\infty})$ the finite dimensional representation associated to $\pi$ by the cohomological property.

We define  $X_{r,s}$ the $GU(r,s)(\R)$ conjugate class of  $$\begin{array}{rcr} h_{r,s}: \mathbb{S}(\C)=\C^{\times}\times \C^{\times}&\rightarrow &GU(r,s)(\C)\cong GL_{n}(\C)\times \C^{\times}\\
(z,\overline{z})&\mapsto & (\begin{pmatrix}
zI_{r} & 0\\
0 & \overline{z}I_{s}
\end{pmatrix},
z\overline{z})
\end{array}.$$

  We know that $(GU(r,s),X_{r,s})$ is a Shimura datum and we denote $Sh(GU(r,s))$ the Shimura variety associated to this Shimura datum. The finite dimensional representation $W(\pi_{\infty})$ defines a complex local system $\mathcal{W}(\pi_{\infty})$ and $l$-adic local system $\mathcal{W}(\pi_{\infty})_{l}$ on $Sh(GU(r,s))$. 
  
As shown in  \cite{harris97}, the cohomology group $\overline{H}^{rs}(Sh(GU(r,s)), \mathcal{W}(\pi_{\infty}))$ defined in section $2.2$ of \cite{harris97} is naturally endowed with a De Rham rational structure and a Betti rational structure over $K$ (c.f. Proposition $2.2.7$ of \textit{loc.cit}). The cohomology group $\overline{H}^{rs}(Sh(GU(r,s)), \mathcal{W}(\pi_{\infty})_{l})$ is endowed with an $l$-adic structure. Moreover, $\pi_{f}$ contributes non trivially to these cohomology groups, i.e.  $\overline{H}^{rs}(Sh(GU(r,s)),*)[\pi_{f}]:=Hom_{G(\AK\text{}_{f})}(\pi_{f},\overline{H}^{rs}(Sh(GU(r,s)),*)\neq 0$ for $*= \mathcal{W}(\pi_{\infty})$ or  $\mathcal{W}(\pi_{\infty})_{l}$.

One direct consequence is that the rationality field of $\pi_{f}$ is a number field (see section $2.6$ of \cite{harris97}). One can then realize $\pi_{f}$ over $E(\pi)$, a finite extension of its rationality field, which is still a number field. For each $(r,s)$, we fix a $\xi_{r,s}$ and a $\pi$ descended from $\Pi\otimes \xi_{r,s}$. We take $E(\Pi)$ a number field contains the $E(\pi)$ for all $(r,s)$. We assume also that $E(\Pi)$ contains $K$.

\medskip

One can show that $\overline{H}^{rs}(Sh(GU(r,s)),*)[\pi_{f}]$ for $*= \mathcal{W}(\pi_{\infty})$ or  $\mathcal{W}(\pi_{\infty})_{l}$ form a motive with coefficients in $E(\Pi)$ (c.f. Proposition $2.7.10$ of \textit{loc.cit}). This is a motive associated to $\Pi$ as in Conjecture \ref{conjmotive}, denoted by $M(\Pi)$. Since $\Pi$ is conjugate self-dual, we have $w_{\Pi}=0$ (c.f. Lemma \ref{wPi}) and then $M(\Pi)$ is pure of weight $n-1$. Also due to the conjugate self-dual property, $M(\Pi)$ is polarized. In particular, there exists $<.>$ a non degenerate bilinear form on $\overline{H}^{rs}(Sh(GU(r,s)),\mathcal{W}(\pi_{\infty}))[\pi_{f}]$ normalized as in section $2.6.8$ of \cite{harris97}.

\bigskip

With the help of coherent cohomology, we can furthermore decompose\\ $\overline{H}^{rs}(Sh(GU(r,s)),\mathcal{W}(\pi_{\infty}))[\pi_{f}]$. Let $Sh(GU(r,s))\hookrightarrow \widetilde{Sh}(GU(r,s))$ be a smooth toroidal compactification. Let $\pi'_{\infty}$ be any discrete series representation of $GU(r,s)(\R)$ with base change $\Pi^{\vee}_{\infty}$. It is then cohomological and we denote $\widetilde{\mathcal{E}}(\pi'_{\infty})$ the coherent automorphic vector bundle attached to the finite dimensional representation associated to $\pi'_{\infty}$. We have $$\overline{H}^{rs}(Sh(GU(r,s)),\mathcal{W}(\pi_{\infty}))[\pi_{f}]=\bigoplus\limits_{\pi_{\infty}'}\widetilde{H}^{q(\pi_{\infty}')}(\widetilde{Sh}(GU(r,s)),\widetilde{\mathcal{E}}(\pi'_{\infty}))[\pi_{f}]$$ where $\widetilde{H}$ indicates the coherent cohomology and $q(\pi_{\infty}')$ is an integer depends on $\pi_{\infty}'$. Among these $\pi_{\infty}'$, there exists a holomorphic representation $\pi'_{\infty}$ such that $q(\pi_{\infty}')=0$. We fix this $\pi_{\infty}'$ and choose a $K$-rational element $0\neq \beta\in \widetilde{H}^{q(\pi_{\infty}')=0}(\widetilde{Sh}(GU(r,s)),\widetilde{\mathcal{E}}(\pi'_{\infty}))[\pi_{f}]$.
At last,  we write $\xi_{r,s,\infty}(t)=t^{C_{r,s}}$ with $C_{r,s}\in \Z$. We define the \textbf{automorphic period} of $(\Pi,\xi_{r,s})$ by:
$$P^{(s)}(\Pi,\xi_{r,s}):=(2\pi)^{-C_{r,s}}<\beta,\beta>.$$ It is a non zero complex number. By the following proposition, we see that $P^{(s)}(\Pi,\xi_{r,s})$ does not depend on the choice of $\beta$ modulo $E(\pi)^{\times}$ and thus well defined modulo $E(\pi)^{\times}$.
\begin{prop}(Proposition $3.19$ in \cite{harrisimrn})

Let $\beta'\in \widetilde{H}^{q(\pi_{\infty}')=0}(\widetilde{Sh}(GU(r,s)),\widetilde{\mathcal{E}}(\pi'_{\infty}))[\pi_{f}]$ be another $K$-rational element. We have $\cfrac{<\beta,\beta'>}{<\beta,\beta>}\in E(\pi)$. 

Consequently, $<\beta,\beta>\sim_{E(\pi)}<\beta',\beta'>$.
\end{prop}

\begin{rem}
We may furthermore choose $\beta=\beta(\Pi)$ such that $<\beta,\beta>$ is equivariant under action of $G_{K}$ in the sense that Equation (\ref{mainformula}) holds.
\end{rem}

Actually, this period $P^{(s)}(\Pi,\xi_{r,s})$ is also independent of the choice of $\xi_{r,s}$. This is a corollary to Theorem \ref{maintheorem}. See the end of section \ref{independentxi} for a further discussion of this point. Meanwhile, we fix a $\xi_{r,s}$ and define the ($s$-th) \textbf{automorphic period} of $\Pi$ by $P^{(s)}(\Pi):=P^{(s)}(\Pi,\xi_{r,s})$.

\begin{rem}\label{supercuspidal}
We now see that one can construct a motive associated to $\Pi$ if $\Pi$ descends for any of these groups $GU(r,s)$. If we replace $GU(r,s)$ by the similitude unitary group of sign $(n-1,1)$ at infinity and quasi-split at every finite places, the condition that $\Pi$ is supercuspidal at places over $q$ is revealed to be unnecessary. On the other hand, if we want to define $P^{(s)}$ for all $s$, then this condition is inevitable in the setting of this article.
\end{rem}

\end{subsection}

\begin{subsection}{Special values of $L$-functions for $GL_{n}\times GL_{1}$}\label{independentxi}

In this section, we consider $\Pi$, a cuspidal representation of $GL_{n}(\AK)$ which is regular, cohomological, conjugate self-dual and supercuspidal at places over $q$.

The main theorem of \cite{harris97} is stated as follows. The original theorem (Theorem $3.5.13$ of \cite{harris97}) assumed the condition $m>n-\cfrac{\kappa}{2}$ (see the following theorem for notation). This condition was improved by M. Harris in several articles later (c.f. section $4.3$ of \cite{harrismotivic}, chapter $3$ of \cite{harrisunitaryperiod} or chapter $4$ of \cite{harrissimpleproof}).

\begin{thm}\label{maintheorem}(Theorem $4.27$ in \cite{harrismotivic})

Let $\Pi$ be as in the beginning of this section. We denote its infinity type by $(z^{a_{i}}\overline{z}^{-a_{i}})_{1\leq i\leq n}$ with $a_{1}>a_{2}>\cdots>a_{n}$, $a_{i}\in \Z+\cfrac{n-1}{2}$ for all $1\leq i\leq n$. Let $\eta$ be an algebraic Hecke character of $K$ with infinity type $\eta_{\infty}(z)=z^{a}\bar{z}^{b}$. We assume that for all $1\leq i \leq n$, $b-a\neq 2a_{i}$.

Write $\eta^{c}=\widetilde{\beta}\alpha$. Here $\alpha$, $\beta$ are Hecke characters of $K$ with $\alpha_{\infty}(z)=z^{\kappa}$ and $\beta_{\infty}(z)=z^{-k}$, $\kappa, k\in \Z$ as in the section \ref{alpha}.  It is easy to verify that $\kappa=a+b$, $k=a$.

Define $s=s({\eta}^{c}, \Pi^{\vee})=\#\{i\mid  a-b+2a_{i}<0\}$ and $r=n-s=\#\{i\mid  a-b+2a_{i}>0\}=\max\{i\mid a-b+2a_{i}>0\}$.

 For $m \in \Z$,  $m$ is critical for $M(\Pi)\otimes M(\eta)$ if and only if
\begin{equation}\label{criticalvalue2}
\max(1-a_{r}-a,1+a_{r+1}-b) \leq m-\cfrac{n-1}{2} \leq \min(a_{r}-b, -a_{r+1}-a).
\end{equation}

If such $m$ satisfies furthermore $m\geq\cfrac{n-\kappa}{2}=\cfrac{n-a-b}{2}$, then
\begin{eqnarray} \label{mainformula}
&L(m,M(\Pi)\otimes M(\eta))=L(m-\cfrac{n-1}{2}, \Pi\otimes \eta)\sim_{E(\Pi)E(\beta)E(\alpha)}&\nonumber\\
 &(2\pi i)^{(m-\frac{n-1}{2})n}\mathcal{G}(\varepsilon_K)^{[\frac{n}{2}]}P^{(s)}(\Pi)
[(2\pi i)^{\kappa}\mathcal{G}(\alpha_{0})]^{s}
[ (2\pi i)^{k}p(\check{\beta}^{(2)}\check{\alpha},1)]^{n-2s}&
\end{eqnarray}
equivariant under action of $G_{K}$.

Here  $\varepsilon_K$ is the Artin character of $\AQ$ associated to the extension $K/\Q$.

\end{thm}

\begin{rem}

\begin{enumerate}

\item $M(\eta)$ is the motive associated to the Hecke character $\eta$ which can be seen as a special case of motives introduced in the previous section. For motives associated to Hecke characters over general fields, detailed introduction can be found in chapter $8$ of \cite{deligne79}.


\item We have applied theorem $4.27$ of \cite{harrismotivic} to $(\Pi^{\vee}, \eta^{c})$ but not $(\Pi,\eta)$. A little calculation on the change of notation for $k$, $\kappa$ and $s$ is needed here.

From the theorem in \textit{loc.cit}, we get an interpretation for $L(m, \Pi^{\vee} \otimes \eta^{c})$ directly. But since $\Pi$ is conjugate self-dual and thus $\Pi^{\vee} \otimes \eta^{c} \cong \Pi^{c}\otimes \eta ^{c}$, we have $L(s,\Pi^{\vee} \otimes \eta^{c})=L(s,\Pi^{c}\otimes\eta^{c})=L(s,\Pi\otimes \eta)$. Therefore we can obtain a formula for $L(m,\Pi\otimes \eta)$ as in the theorem.

\item We have already seen that $M(\Pi)\otimes M(\eta)$ is pure. The condition $b-a\neq 2a_{i}$ for all $1\leq i \leq n$ implies that the Hodge type of $M(\Pi)\otimes M(\eta)$ doesn't contain any type of form $(p,p)$ and then satisfies the condition in Lemma \ref{critical}. The result above on the critical values thus follows from Lemma \ref{critical}.

We remark that if $a-b+2a_{i}=0$ for some $i$, there is no critical value for $M(\Pi)\otimes M(\eta)$ (c.f. Lemma $1.7.1$ of \cite{harris97}).

\item The infinity type of $\Pi$, $(z^{a_{i}}\overline{z}^{-a_{i}})_{1\leq i\leq n}$, is invariant under the action of $G_{K}$.

In fact, there exists a motive $M_{0}$ over $K$ with coefficients in $E$ associated to $\Pi$ and the motive $M(\Pi)$ considered above is $Res_{K/Q}M_{0}$. For the fixed $K\hookrightarrow \C$, $M_{0,B}\otimes \C$ has a Hodge decomposition $M_{0,B}\otimes \C=\bigoplus\limits_{p,q\in \Z}M^{p,q}$ as $E\otimes \C$-module. It has Hodge type $\{(-a_{i}+\cfrac{n-1}{2},-b_{i}+\cfrac{n-1}{2})|1\leq i\leq n\}$ (c.f. paragraph $4.3.3$ of \cite{clozelaa}).

Let $\sigma$ be an element of $G_{K}$. We may identify $\sigma$ with $\sigma|_{E}$, an embedding from $E$ to $\C$, when considering $\Pi^{\sigma}$. We know that Remark \ref{infinitytype} also applies to $M_{0}$. In another word, the Hodge type of $M_{0,B,\sigma}$ is independent of $\sigma$. Since $\sigma$ fixes $K$, it does not change the fixed embedding $K\hookrightarrow \C$. Therefore, the equation $L(s,M_{0},\sigma)=L(s,\Pi^{\sigma})$ implies that $\Pi^{\sigma}$ also has the Hodge type $\{(-a_{i}+\cfrac{n-1}{2},-b_{i}+\cfrac{n-1}{2})|1\leq i\leq n\}$. We can now conclude that the infinity type of $\Pi^{\sigma}$ is the same with $\Pi$.

Similarly, the infinity type of $\eta$ is also invariant under action of $G_{K}$. In particular, the index $s=s({\eta}^{c}, \Pi^{\vee})=\#\{i\mid  a-b+2a_{i}<0\}$ is invariant under the action of $G_{K}$.


\end{enumerate}
\end{rem}

\begin{prop}\label{nonvanishing}
Let $\Pi$ be as in Theorem \ref{maintheorem}. We preserve the notation in Theorem \ref{maintheorem}.

\begin{enumerate}
\item\label{nonvanishingm} For any algebraic Hecke character $\eta$ where its infinity type $z^{a}\bar{z}^{b}$ satisfies $b-a\neq 2a_{i}$ for all $i$, there exists an integer $m\geq \cfrac{n-\kappa}{2}=\cfrac{n-a-b}{2}$ such that $m$ is critical for $M(\Pi)\otimes M(\eta)$.

\item For any fixed integer $0\leq s\leq n$, there exists an algebraic Hecke character $\eta$ and an integer $m$ as in (\ref{nonvanishingm}) such that $a-b+2a_{i}\neq 0$ for all $i$, $s=\#\{i\mid a-b+2a_{i}<0\}$ and $L(m,M(\Pi)\otimes M(\eta))\neq 0$.

where $a,b$ are indices for the infinity type of $\eta$ as before.\end{enumerate}
\end{prop}

\begin{dem}

\begin{enumerate}
\item Let $\eta$ be a fixed Hecke character and $s$ be the integer $\#\{i\mid a-b+2a_{i}<0\}$. Then $r=n-s=\#\{i\mid a-b+2a_{i}>0\}=max\{i\mid a-b+2a_{i}>0\}$. Therefore, $-1-2a_{r+1}\geq a-b \geq 1-2a_{r}$ by the fact that $a-b+2a_{i}\neq 0$ for all $i$.

Moreover, the two integers $\cfrac{n-1}{2}+a_{r}-b$ and $\cfrac{n-1}{2}-a_{r+1}-a$ are both no less than $\cfrac{n-a-b}{2}$. Hence there exists an integer $m$ such that $$\cfrac{n-a-b}{2}\leq m\leq \min(a_{r}-b, -a_{r+1}-a)+\cfrac{n-1}{2}. $$

To prove that $m$ is critical for $M(\Pi)\otimes M(\eta)$, it suffice to verify that $\max(1-a_{r}-a,1+a_{r+1}-b)+\cfrac{n-1}{2}\leq \cfrac{n-a-b}{2}$ which is trivial.

\item Let $1\leq r\leq n-1$ be a fixed integer now. We may take $a$, $b\in \Z$ such that $-1-2a_{r+1}\geq a-b \geq 1-2a_{r}$. Then $a_b+2a_{i}\neq 0$ for all $i$ and $r=\max\{i\mid a-b+2a_{i}>0\}$.

When $a_{r}-a_{r+1}\geq 2$, there exists an integer $m$ as in (\ref{nonvanishingm}) with non vanishing $L$-value if $\alpha$ satisfies condition $(3.2.1)$ of \cite{harrisunitaryperiod} (see
Comments ($3.2.2.ii$) and corollary $3.3$ of \cite{harrisunitaryperiod}). Here $\eta^{c}=\tilde{\beta}\alpha$ as in Theorem \ref{maintheorem}. When $a_{r}-a_{r+1}= 1$, such $m$ exists provided that $\alpha$ satisfies more conditions on finite number of finite places (see Theorem $3.4$ of \cite{harrisunitaryperiod}). In all cases, we can find $\eta$ with infinity type $z^{a}\bar{z}^{b}$ and an integer $m$ as in (\ref{nonvanishingm}) such that $L(m,M(\Pi)\otimes M(\eta))$ does not vanish.
\end{enumerate}
\end{dem}

For any $0\leq s\leq n$ fixed, we can now show that the period $P^{(s)}(\Pi,\xi_{r,s})$ depends only on $\Pi$. We fix a Hecke character $\eta$ and an integer $m\geq \cfrac{n-\kappa}{2}$ as in second part of Proposition \ref{nonvanishing}. We apply Theorem \ref{maintheorem} to $(\Pi,\eta)$ and $m$. Note that both the left hand side and the right hand side of equation (\ref{mainformula}) are non zero. Hence the transitivity holds forthe relation $\sim$ now (c.f. Remark \ref{transitive}). Therefore, for any $\xi_{r,s}'$ verifying conditions in Theorem \ref{bcsimilitude}, we have $P^{(s)}(\Pi,\xi_{r,s})\sim_{E(\Pi)E(\beta)E(\alpha)} P^{(s)}(\Pi,\xi_{r,s}') $. On varying $\eta$, we conclude that:

\begin{cor}\label{independence}
The period $P^{(s)}(\Pi,\xi_{r,s})$ does not depend on the choice of $\xi_{r,s}$ modulo $E(\Pi)^{\times}$.
\end{cor}

\end{subsection}

\begin{subsection}{Special values of $L$-functions for $GL_{n}\times GL_{n-1}$}\label{ntimesn-1}

In \cite{harrismotivic}, the authors also gave an interpretation of special values of $L$-function for $GL_{n}\times GL_{n-1}$ over a quadratic imaginary field in terms of automorphic periods. Let us introduce this result in this section.

Let $n\geq 2$ be an integer and $K$ be a quadratic imaginary field.

Let $\Pi$ and $\Pi'$ be two cuspidal representations of $GL_{n}(\AK)$ and $GL_{n-1}(\AK)$ which satisfy all the conditions in Theorem \ref{maintheorem}, i.e. $\Pi$ and $\Pi'$ are both regular, cohomological, conjugate self-dual and supercuspidal at places over $q$.

We denote the infinity type of $\Pi$ by $(z^{a_{i}}\overline{z}^{-a_{i}})_{1\leq i\leq n}$ with $a_{1}>a_{2}>\cdots>a_{n}$. Since $\Pi$ is cohomological, it is associated to a finite dimensional representation.  The highest weight $(\mu_{1},\mu_{2}, \cdots,\mu_{n})$ of this finite dimensional representation is calculated in section $2.4$ of \cite{harrismotivic}:
$$\mu_{1}=-a_{n}-\frac{n-1}{2}, \mu_{2}=-a_{n-1}-\frac{n-3}{2},\mu_{3}=-a_{n-2}-\frac{n-5}{2}\cdots, \mu_{n}=-a_{1}+\frac{n-1}{2}.$$

Similarly, if we denote the infinity type of $\Pi'$ by $(z^{b_{i}}\overline{z}^{-b_{i}})_{1\leq i\leq n-1}$ with $b_{1}>b_{2}>\cdots>b_{n-1}$ and the highest weight of the finite dimensional representation associated to it by $(\lambda_{i})_{1\leq i\leq n-1}$, we have:
$$\lambda_{1}=-b_{n-1}-\frac{n-2}{2}, \lambda_{2}=-b_{n-2}-\frac{n-4}{2},\lambda_{3}=-b_{n-3}-\frac{n-6}{2}\cdots, \lambda_{n-1}=-b_{1}+\frac{n-2}{2}.$$

The method in \cite{harrismotivic} requires the following hypothesis:
\begin{hyp}\label{hyp}
The highest weight $\mu$ and $\lambda$ satisfy:
$$\mu_{1}\geq -\lambda_{n-1}\geq \mu_{2}\geq -\lambda_{n-2}\geq \cdots\geq -\lambda_{1}\geq \mu_{n}.$$
\end{hyp}

In \cite{harrismotivic}, the authors defined:
\begin{itemize}
\item $p(m,\Pi_{\infty},\Pi'_{\infty})$, a complex number depending only on $m$, $\Pi_{\infty}$, $\Pi'_{\infty}$(c.f. Proposition $6.4$ of \textit{loc.cit});

\item $Z(\Pi_{\infty})$ (resp. $Z(\Pi'_{\infty})$) a  complex number depending only on $\Pi_{\infty}$ (resp. $\Pi'_{\infty}$) (c.f. Theorem $6.7$ of \textit{loc.cit}).

\end{itemize}

We also need the following hypothesis:

\begin{hyp}\label{hypregular}
$\Pi$ and $\Pi'$ are \textbf{very regular}, which means $\mu_{j}-\mu_{j+1}\geq 2$ for all $j$ and $\lambda_{k}-\lambda_{k+1}\geq 2$ for all $k$. In other words, $a_{j}-a_{j+1}\geq 3$ for all $j$ and $b_{k}-b_{k+1}\geq 3$ for all $k$.
\end{hyp}

We can then restate Theorem $6.10$ of \cite{harrismotivic}:
\begin{thm}\label{maintheorem2}
Let $\Pi$ and $\Pi'$ be as above. We assume the highest weight $\mu$ and $\lambda$ defined above satisfy the Hypothesis \ref{hyp} and \ref{hypregular}.
Let $m$ be a non negative integer. If $m+n-1$ is critical for $M(\Pi)\otimes M(\Pi')$, then
$$L(m+\frac{1}{2},\Pi\times \Pi')\sim_{E(\Pi)E(\Pi')} p(m,\Pi_{\infty},\Pi'_{\infty})Z(\Pi_{\infty})Z(\Pi'_{\infty})\prod\limits_{j=1}^{n-1}P^{(j)}(\Pi)\prod\limits_{k=1}^{n-2}P^{(k)}(\Pi')$$
equivariant under action of $G_{K}$.
\end{thm}

\end{subsection}
\end{section}

\begin{section}{Period relations for automorphic induction of Hecke characters}
In this section, we show the functoriality of automorphic period for cyclic automorphic induction of Hecke characters.
\begin{subsection}{First calculation}\label{etaandm}
We preserve the assumptions for $\Pi$ and notation in Theorem \ref{maintheorem}.

We write $$\begin{array}{lcl}

\text{A} & := &  [(2\pi i)^{\kappa}\mathcal{G}(\alpha_{0})]^{s} p(\check{\alpha},1)^{n-2s};\\

\text{B} & := &  [ (2\pi i)^{k}p(\check{\beta}^{(2)},1)]^{n-2s}.

\end{array}$$

By the Lemma \ref{lemmas}, we have
$$\begin{array}{lcl}

(A)_{\sigma\in \Sigma_{E(\alpha)}} & \sim_{E(\alpha)} & [p(\alpha\alpha^{c},1)^{-s}p(\alpha^{-1,c},1)^{n-2s}] _{\sigma\in \Sigma_{E(\alpha)}} \\
&\sim _{E(\alpha)} &[p(\alpha^{-1}, 1)^{s}p(\alpha^{-1,c},1)^{n-s}]_{\sigma\in \Sigma_{E(\alpha)}} \\
&\sim_{E(\alpha)}& [p(\alpha^{-1}, 1)^{s}p(\alpha^{-1},\iota)^{n-s}]_{\sigma\in \Sigma_{E(\alpha)}};\\

(B)_{\sigma\in \Sigma_{E(\beta)}} & \sim_{E(\beta)} & [p(\widetilde{\beta},1)^{n-2s}]_{\sigma\in \Sigma_{E(\beta)}} \\
& \sim_{E(\beta)}  & [p(\widetilde{\beta}^{-1},1)^{s} p(\widetilde{\beta},1)^{n-s}]_{\sigma\in \Sigma_{E(\beta)}}  \\
& \sim_{E(\beta)}  & [p(\widetilde{\beta}^{-1},1)^{s} p(\widetilde{\beta}^c,\iota)^{n-s}]_{\sigma\in \Sigma_{E(\beta)}}  \\
&\sim_{E(\beta)} &  [p(\widetilde{\beta}^{-1},1)^{s} p(\widetilde{\beta}^{-1},\iota)^{n-s} ]_{\sigma\in \Sigma_{E(\beta)}} .

\end{array}$$

Therefore,
$$\begin{array}{lcl}
(AB)_{\sigma\in \Sigma_{E(\alpha)E(\beta)} }&\sim_{E(\alpha)E(\beta)}& [p(\widetilde{\beta}^{-1}\alpha^{-1},1)^{s}p(\widetilde{\beta}^{-1}\alpha^{-1},\iota)^{n-s} ]_{\sigma\in \Sigma_{E(\alpha)E(\beta)} }  \\
&\sim_{E(\alpha)E(\beta)}& [p(\check{\eta},1)^{s}p(\check{\eta},\iota)^{n-s}]_{\sigma\in \Sigma_{E(\alpha)E(\beta)} }.
\end{array}$$

We apply moreover Proposition \ref{nonvanishing}. We can then deduce that:

\begin{cor}
Let $0\leq s\leq n$ be a fixed integer.
Let $\Pi$ be as in Theorem \ref{maintheorem}. There exists $\eta$ an algebraic Hecke character of $K$ and $m\in \Z+\cfrac{n-1}{2}$ with $a-b+2a_{i}\neq 0$ for all $i$, $s=\#\{i\mid a-b+2a_{i}<0\}$, $m+\cfrac{n-1}{2}\geq \cfrac{n-\kappa}{2}=\cfrac{n-a-b}{2}$ critical for $M(\Pi\otimes \eta)$ such that
\begin{equation}\label{maineqn}
L(m, \Pi\otimes \eta)\sim_{E(\Pi)E(\eta);K}   (2\pi i)^{mn}\mathcal{G}(\varepsilon_K)^{[\frac{n}{2}]}P^{(s)}(\Pi)
p(\check{\eta},1)^{s}p(\check{\eta},\iota)^{n-s}
\end{equation}
equivariant under action of $G_{K}$ and with both side non zero.

Here $a,b$ are indices for the infinity type of $\eta$ as in Theorem \ref{maintheorem}, and $\varepsilon_K$ is the Artin character of $\AQ$ associated to the extension $K/\Q$.
\end{cor}

\begin{rem}
\begin{enumerate}
\item Here we have shifted $m$ by $\cfrac{n-1}{2}$ to simplify the notation in the following.
\item Note that we have replaced $E(\alpha)E(\beta)$ by $E(\eta)$. This can be done by varying $\alpha$ and $\beta$.
\end{enumerate}
\end{rem}

\end{subsection}
\begin{subsection}{Odd dimensional case}
\text{}

Let $F=F^{+}K$ with $F^{+}$ a totally real field of degree $n$ over $\Q$ and $K$ a quadratic imaginary field as in section \ref{notation}.  We assume $F$ is cyclic over $K$. We fix $\{\sigma_{1},\sigma_{2},\cdots, \sigma_{n}\}\subset \Sigma_{F}$,  a CM type of $F$, consisting of elements which are trivial on $K$. 

Let $\chi$ be an algebraic conjugate self-dual Hecke character of $F$. Here conjugate self-dual means $\chi^{-1}=\chi^{c}$. The infinity type of $\chi$ is thus of the form $\chi_{\infty}(z)=\prod\limits_{i=1}^{n}\sigma_{i}\left(\cfrac{z}{\overline{z}}\right)^{a_{i}}$ with $a_{i}\in \Z$. We assume $\chi$ is \textbf{regular}, namely, $a_{i}\neq a_{j}$ for $i\neq j$.

Moreover, we may assume $\chi \neq \chi^{\tau}$ for all $\tau\in Gal(F/K)$ non trivial. Otherwise by Theorem \ref{bccharacter}, $\chi$ is the base change of a Hecke character of a smaller CM field. In this case, we replace $F$ by this smaller CM field.

We assume at first that $n$ is odd.

We denote by $\Pi(\chi)$ the automorphic induction of $\chi$ from $GL_{1}(\AF)$ to $GL_{n}(\AK)$. It is a cuspidal cohomological representation of $GL_{n}(\AK)$ (c.f. \cite{arthurclozel} chapter $3$ Theorem $6.2$). The cuspidal conclusion comes from the fact that  $\chi \neq \chi^{\tau}$ for all $1\neq \tau\in Gal(F/K)$. This fact is not stated but proved in the proof of Theorem $6.2$ in chapter $3$ of \cite{arthurclozel}. The cohomology conclusion, in particular the algebraic conclusion, is due to the fact that $n$ is odd. Moreover, since $\chi$ is conjugate self-dual, we know $\Pi(\chi)$ is also conjugate self-dual. Its infinity type is $(z^{a_{i}}\overline{z}^{-a_{i}})_{1\leq i \leq n}$.

Up to finite extension, we may assume $E(\Pi(\chi))=E(\chi)$. In particular, we assume $E(\chi)$ contains $K$. We also make the hypothesis that:

 \begin{hyp}\label{hypsupercuspidal}
 For any $v$ a place of $K$ over $q$, $\chi_{v}\neq \chi_{v}^{\tau}$ for all $\tau\in Gal(F_{v}/K_{v})$.
  \end{hyp}

 Under this hypothesis, $\Pi(\chi)$ is supercuspidal at all places over $q$ (c.f. Proposition $2.4$ of \cite{harrislanglands}).

We have just shown that the representation $\Pi:=\Pi(\chi)$ satisfies all the conditions in Theorem \ref{maintheorem}. Let $\eta$ and $m$ be as in the last part of Section \ref{etaandm}. We then have (see equation (\ref{maineqn})):

\begin{equation}\label{mainequation}
L(m, \Pi(\chi)\otimes \eta) \sim_{E(\chi)E(\eta);K} (2\pi i)^{mn}\mathcal{G}(\varepsilon_K)^{[\frac{n}{2}]}P^{(s)}(\Pi(\chi))
p(\check{\eta},1)^{s}p(\check{\eta},\iota)^{n-s} 
\end{equation}
 equivariant under action of $G_{K}$ and with both sides non zero.
\bigskip

On the other hand, we have
\begin{equation}\label{LfunctionAI}
L(m,\Pi(\chi)\otimes \eta)=L(m,\chi\otimes \eta\circ N_{\AF/\AK})
\end{equation}
 by the automorphic induction.

Now let us consider $\chi\otimes \eta\circ N_{\AF/\AK}$. It has infinity type $\prod\limits_{1\leq i\leq n} \sigma_{i}^{a+a_{i}}\bar{\sigma}_{i}^{b-a_{i}}$. It is critical since $a-b+2a_{i}\neq 0$ for all $i$.

We define $\Phi_{s,\chi}$ a CM type of $F$ by $\sigma_{i}\in \Phi_{s,\chi}$ if $a-b+2a_{i}<0$; $\bar{\sigma}_{i}\in \Phi_{s,\chi}$ if $a-b+2a_{i}>0$. This CM type is compatible with $\chi\otimes \eta\circ N_{\AF/\AK}$. The notation $\Phi_{s,\chi}$ implicitly indicates that this set depends only on $s$ and $\chi$, but not on $\eta$. Indeed, we can easily deduce from the fact $s=\#\{i\mid  a-b+2a_{i}<0\}$ that for each $i$ such that $a_{i}$ is one of the $s$ smallest numbers in $\{a_{i},1\leq i\leq n\}$, we have $\sigma_{i}\in \Phi_{s,\chi}$; otherwise $\bar{\sigma}_{i}\in \Phi_{s,\chi}$. Moreover, $\#\Phi_{s,\chi}\cap\{\sigma_{i},1\leq i \leq n\}=s$. In other words, there are exactly $s$ elements in $\Phi_{s,\chi}$ whose restriction to $K$ is trivial.

Comparing equation (\ref{criticalvalue1}) with equation (\ref{criticalvalue2}), we see that $m+\cfrac{n-1}{2}$ is critical for $M(\Pi(\chi))\otimes M(\eta)$ if and only if $m$ is critical for $\chi\otimes \eta\circ N_{\AF/\AK}$. Hence $m$ is critical for $\chi\otimes \eta\circ N_{\AF/\AK}$ and we may apply Blasius's result here:

$$L(m,\chi\otimes \eta\circ N_{\AF/\AK})\sim_{E(\chi)E(\eta)}
 D_{F^{+}}^{1/2}(2\pi i)^{mn} p((\chi\otimes \eta\circ N_{\AF/\AK})^{c,-1},\Phi_{s,\chi})$$
 
 equivariant under action of $G_{K}$.
 \bigskip

Therefore, by Proposition \ref{propCM}, we have the following equations equivariant under action of $G_{K}$:
 $$
\begin{array}{lcl}
& & L(m,\chi\otimes \eta\circ N_{\AF/\AK})\\
&\sim_{E(\chi)E(\eta)} & D_{F^{+}}^{1/2}(2\pi i )^{mn} p(\check{\chi},\Phi_{s,\chi})p((\eta\circ N_{\AF/\AK})^{c,-1},\Phi_{s,\chi})\\
&\sim_{E(\chi)E(\eta)} & D_{F^{+}}^{1/2}(2\pi i )^{mn} p(\check{\chi},\Phi_{s,\chi})p(\check{\eta}\circ N_{\AF/\AK},\Phi_{s,\chi})\end{array}.
$$

Still by Proposition \ref{propCM}, we have: $$\begin{array}{lcl}
& & p(\check{\eta}\circ N_{\AF/\AK},\Phi_{s,\chi})\\
 &\sim_{E(\eta)}& \prod\limits_{1\leq i \leq n, \sigma_{i}\in \Phi_{s,\chi}}p(\check{\eta}\circ N_{\AF/\AK},\sigma_{i})\prod\limits_{1\leq i \leq n, \bar{\sigma}_{i}\in \Phi_{s,\chi}}p(\check{\eta}\circ N_{\AF/\AK},\bar{\sigma}_{i})\\
  &\sim_{E(\eta)}& \prod\limits_{1\leq i \leq n, \sigma_{i}\in \Phi_{s,\chi}}p(\check{\eta} ,\sigma_{i}|_{K})\prod\limits_{1\leq i \leq n, \bar{\sigma}_{i}\in \Phi_{s,\chi}}p(\check{\eta},\bar{\sigma}_{i}|_{K})\\

&\sim _{E(\eta)}& p(\check{\eta},1)^{s} p(\check{\eta},\iota)^{n-s}
\end{array}$$
equivariant under action of $G_{K}$.

We then get
\begin{equation}
L(m,\chi\otimes \eta\circ N_{\AF/\AK})
\sim_{E(\chi)E(\eta)} D_{F^{+}}^{1/2}(2\pi i )^{mn}p(\check{\chi},\Phi_{s,\chi})p(\check{\eta},1)^{s} p(\check{\eta},\iota)^{n-s}
\end{equation}
equivariant under action of $G_{K}$ and with both sides non zero.

Comparing this formula with (\ref{mainequation}) and (\ref{LfunctionAI}), we deduce at last that equivariant under action of $G_{K}$:
$$P^{(s)}(\Pi(\chi))\sim_{E(\chi)E(\eta)} D_{F^{+}}^{1/2}\mathcal{G}(\varepsilon_{K})^{-[\frac{n}{2}]}p(\check{\chi},\Phi_{s,\chi})$$

Varying $\eta$, we can replace the field $E(\chi)E(\eta)$ by $E(\chi)$ which is easy to see from Lemma \ref{anotherdef}. We obtain at last:

\begin{thm}
Let $n$ be an odd integer. Let $F=F^{+}K$ with $F^{+}$ a totally real field of degree $n$ over $\Q$ and $K$ a quadratic imaginary field as before. Assume that $F$ is cyclic over $K$. Let $\chi$ be a regular conjugate self-dual algebraic Hecke character of $F$ satisfying Hypothesis \ref{hypsupercuspidal}. Put $\Pi=\Pi(\chi)$ the automorphic induction of $\chi$ from $GL_{1}(\AF)$ to $GL_{n}(\AK)$. The automorphic period of $\Pi$ defined in section \ref{automorphicperiod} satisfies:

\begin{equation}\label{oddfinal}
P^{(s)}(\Pi) \sim_{E(\chi)} D_{F^{+}}^{1/2}\mathcal{G}(\varepsilon_{K})^{-[\frac{n}{2}]}p(\check{\chi},\Phi_{s,\chi})
\end{equation}
equivariant under action of $G_{K}$.
\end{thm}
\end{subsection}

\begin{subsection}{Even dimensional case}\label{evendim}

Now let $n$ be an even positive integer.  Let $F=F^{+}K$ with $F^{+}$ a totally real field of degree $n$ over $\Q$. We still assume $F$ is cyclic over $K$.

Let $\chi$ be an algebraic conjugate self-dual Hecke character of $F$ with infinity type $\chi_{\infty}(z)=\prod\limits_{i=1}^{n}\sigma_{i}\left(\cfrac{z}{\overline{z}}\right)^{a_{i}}$, $a_{i}\in \Z$. We assume also $\chi \neq \chi^{\tau}$ for all $1\neq \tau\in Gal(F/K)$ and the Hypothesis \ref{hypsupercuspidal}

 It follows that $\Pi(\chi)$, the automorphic induction of $\chi$ from $GL_{1}(\AF)$ to $GL_{n}(\AK)$, is still cuspidal conjugate self-dual. It has infinity type $(z^{a_{i}}\overline{z}^{-a_{i}})_{1\leq i\leq n}$. Since $n$ is even, $\Pi(\chi)$ is no longer algebraic. We then consider $\Pi(\chi)\otimes ||\cdot||_{\AK}^{-\frac{1}{2}}$. It is algebraic but not conjugate self-dual.

 By Lemma \ref{psi}, we may take $\psi$ an algebraic Hecke character of $\Q$ such that $\psi\psi^{c}=||\cdot||_{\AK}$ with infinity type $z^{1}\overline{z}^{0}$ and unramified at all places over $q$. We define $\Pi:=\Pi(\chi)\otimes ||\cdot||_{\AK}^{-\frac{1}{2}}\psi$. The infinity type of $\Pi$ is $(z^{a_{i}+\frac{1}{2}}\overline{z}^{-a_{i}-\frac{1}{2}})_{1\leq i\leq n}$. It is a regular, cohomological, conjugate self-dual and supercuspidal at places dividing $q$. In one word, $\Pi$ satisfies all the conditions in Theorem \ref{maintheorem}. We may assume $E(\Pi)=E(\chi)E(\psi)$ up to finite extension.

Let $\eta$ and $m$ be as in the last part of Section \ref{etaandm}. We have (see equation (\ref{maineqn})):

\begin{equation}\label{mainequation2}
L(m+\frac{1}{2}, \Pi\otimes \eta)
\sim_{E(\chi)E(\psi)E(\eta)}(2\pi i)^{mn+\frac{n}{2}}\mathcal{G}(\varepsilon_K)^{[\frac{n}{2}]}P^{(s)}(\Pi)
p(\check{\eta},1)^{s}p(\check{\eta},\iota)^{n-s}\end{equation}
equivariant under action of $G_{K}$ and with both sides non zero.

On the other hand, we have
\begin{eqnarray}\label{LfunctionAI2}
L(m+\frac{1}{2},\Pi\otimes \eta) &=& L(m+\frac{1}{2},\Pi(\chi)\otimes ||\cdot||_{\AF}^{-\frac{1}{2}}\psi\eta) \nonumber\\
&=&L(m,\Pi(\chi)\otimes \psi\eta)\nonumber \\
&=&L(m,\chi\otimes ((\psi\eta)\circ N_{\AF/\AK})).
\end{eqnarray}

Note that the character $\chi\otimes ((\psi\eta)\circ N_{\AF/\AK})$ has infinity type $\prod\limits_{1\leq i\leq n} \sigma_{i}^{a+a_{i}+1}\bar{\sigma}_{i}^{b-a_{i}}$. It is critical under the assumption that $a-b+2a_{i}+1\neq 0$ for all $i$.

We define $\Phi_{s,\chi}$ a CM type of $F$ by $\sigma_{i}\in \Phi_{s,\chi}$ if $a-b+2a_{i}+1<0$; $\bar{\sigma}_{i}\in \Phi_{s,\chi}$ if $a-b+2a_{i}+1>0$. It is compatible with $\chi\otimes ((\psi\eta)\circ N_{\AF/\AK})$. It depends only on $s$ and $\chi$. Moreover, $\Phi_{s,\chi}\cap\{\sigma_{i},1\leq i \leq n\}$ has $s$ elements. We refer to the previous section for an explanation of this fact.

We now apply Blasius's result and get:
$$L(m,\chi\otimes ((\psi\eta)\circ N_{\AF/\AK}) \sim_{E(\chi)E(\psi)E(\eta)}
 D_{F^{+}}^{1/2}(2\pi i)^{mn} p((\chi\otimes ((\psi\eta)\circ N_{\AF/\AK}))^{c,-1},\Phi_{s,\chi})$$
 equivariant under action of $G_{K}$.

Apply Proposition \ref{propCM}, $$
L(m,\chi\otimes ((\psi\eta)\circ N_{\AF/\AK}))
\sim_{E(\chi)E(\psi)E(\eta)}$$
$$
D_{F^{+}}^{1/2}(2\pi i )^{mn} p(\check{\chi},\Phi_{s,\chi})p(\check{\psi}\circ N_{\AF/\AK},\Phi_{s,\chi})p(\check{\eta}\circ N_{\AF/\AK},\Phi_{s,\chi})
$$
equivariant under action of $G_{K}$.

By Proposition \ref{propCM} and the same calculation as in the previous section, we have $$\begin{array}{lcl}
& & p(\check{\psi}\circ N_{\AF/\AK},\Phi_{s,\chi})\sim _{E(\eta)} p(\check{\psi},1)^{s} p(\check{\psi},\iota)^{n-s}\\
&\text{ and }& p(\check{\eta}\circ N_{\AF/\AK},\Phi_{s,\chi})\sim _{E(\eta)} p(\check{\eta},1)^{s} p(\check{\eta},\iota)^{n-s}
\end{array}$$
equivariant under action of $G_{K}$.

We then get
\begin{eqnarray}
&L(m,\chi\otimes ((\psi\eta)\circ N_{\AF/\AK})) \sim_{E(\chi)E(\psi)E(\eta)} &\nonumber\\
&D_{F^{+}}^{1/2}(2\pi i )^{mn}p(\check{\chi},\Phi_{s,\chi})p(\check{\psi},1)^{s} p(\check{\psi},\iota)^{n-s}p(\check{\eta},1)^{s} p(\check{\eta},\iota)^{n-s}&\end{eqnarray}
equivariant under action of $G_{K}$.

Comparing this formula with 	(\ref{mainequation2}) and (\ref{LfunctionAI2}), we deduce at last that equivariant under action of $G_{K}$:
$$
P^{(s)}(\Pi(\chi))
\sim_{E(\chi)E(\psi)E(\eta)} D_{F^{+}}^{1/2}(2\pi i)^{-\frac{n}{2}}\mathcal{G}(\varepsilon_{K})^{-[\frac{n}{2}]}p(\check{\chi},\Phi_{s,\chi})p(\check{\psi},1)^{s} p(\check{\psi},\iota)^{n-s}
$$

Note $p(\check{\psi},1)$ simply by $p(\psi)$ and then $p(\check{\psi},\iota)=p(\check(\psi^{c}),1)=p(\psi^{c})$. Varying $\eta$, we obtain at last:

\begin{thm}
Let $n$ be an even integer. Let $F=F^{+}K$ with $F^{+}$ totally real of degree $n$ over $\Q$ and $K$ a quadratic imaginary field as before. Assume that $F$ is cyclic over $K$. Let $\chi$ be a regular conjugate self-dual algebraic Hecke character of $F$. Put $\Pi=\Pi(\chi)\otimes||\cdot||_{K}^{-\frac{1}{2}}\psi$ where $\psi$ is a Hecke character of $K$ such that $\psi\psi^{c}=||\cdot||_{\AK}$ and $\psi_{\infty}(z)=z$. The automorphic period of $\Pi$ defined in section \ref{automorphicperiod} then satisfies:

\begin{equation}\label{evenfinal}
P^{(s)}(\Pi) \sim_{E(\chi)E(\psi)} D_{F^{+}}^{1/2}(2\pi i)^{-\frac{n}{2}}\mathcal{G}(\varepsilon_{K})^{-[\frac{n}{2}]}p(\check{\chi},\Phi_{s,\chi})p(\psi)^{s}p(\psi^{c})^{n-s}
\end{equation}
equivariant under action of $G_{K}$.
\end{thm}

\begin{rem}
Notice that $M(\Pi)(\frac{n}{2}) =M(\chi)\otimes M(\psi)$. The above result is compatible with M. Harris's calculation on relations of periods by twisting of Hecke characters (see section $2.9$ of \cite{harris97} equation $(2.9.15)$).
\end{rem}

\end{subsection}
\end{section}

\begin{section}{Application: simplification of Archimedean local factors}
We can now refine the Archimedean local factors in \ref{maintheorem2} first in the case $\Pi$ and $\Pi'$ come from a Hecke character and then for general $\Pi$ and $\Pi'$.

Let $F^{+}$ and $F'^{+}$ be two totally real fields of degree $n$ and $n-1$ respectively over $\Q$ and set $F=KF^{+}$, $F'=KF'^{+}$ where $K$ is always a quadratic imaginary field. We assume that $F$ and $F'$ are cyclic over $K$. Put $L=FF'$. It is easy to see that $L$ is a CM field of degree $n(n-1)$ over $K$.

Recall that $\sigma_{1},\sigma_{2},\cdots,\sigma_{n}$ are the complex embeddings of $F$ which are trivial on $K$. Similarly, we denote by $\sigma'_{1},\sigma'_{2},\cdots,\sigma'_{n-1}$ the complex embeddings of $F$ which are trivial on $K$. For every $(j,k)$ with $1\leq j\leq n$ and $1\leq k\leq n-1$, we define $\tau_{jk}$ a complex embedding of $L=FF'$ such that $\tau_{jk}|_{F}=\sigma_{j}$ and $\tau_{jk}|_{F'}=\sigma'_{k}$. Hence $\{\tau_{jk}|1\leq j\leq n, 1\leq k\leq n-1\}$ are all the complex embeddings of $L$ which are trivial on $K$.

We take $\chi$ and $\chi'$ two algebraic regular conjugate self-dual Hecke characters of $F$ and $F'$.  We assume both $\chi$ and $\chi'$ satisfy the Hypothesis \ref{hypsupercuspidal}. Denote the infinity type of $\chi$ by $\prod\limits_{1\leq j\leq n}\sigma_{j}^{a_{j}}\overline{\sigma}_{j}^{-a_{j}}$ with $a_{1}>a_{2}>\cdots >a_{n}$ and the infinity type of $\chi'$ by  $\prod\limits_{1\leq k\leq n-1} \sigma_{k}'^{a_{k}'}\overline{\sigma_{k}'}^{-a_{k}'}$ with $a_{1}'>a_{2}'>\cdots>a_{n}'$. We assume $\chi$ and $\chi'$ are \textbf{very regular} which means that $a_{j}-a_{j+1}\geq 3$ for all $j$ and $a_{k}'-a_{k+1}'\geq 3$ for all $k$.

We define $\Pi:=\Pi(\chi)$ and $\Pi':=\Pi(\chi')\otimes ||\cdot||_{\AK}^{-\frac{1}{2}}\psi$ if $n$ is odd; $\Pi:=\Pi(\chi)\otimes ||\cdot||_{\AK}^{-\frac{1}{2}}\psi$ and $\Pi':=\Pi(\chi')$ if $n$ is even where $\psi$ is a character of $K$ defined in Lemma \ref{psi}.  We assume moreover $(\Pi,\Pi')$ satisfies the Hypothesis \ref{hyp}. We can then apply Theorem \ref{maintheorem2} to $(\Pi,\Pi')$.

\begin{subsection}{The case $n$ odd}

Let us assume firstly that $n$ is odd. In this case, the infinity type of $\Pi$ is $(z^{a_{j}}\overline{z}^{-a_{j}})_{1\leq j\leq n}$ and the infinity type of $\Pi'$ is $(z^{a_{k}+\frac{1}{2}}\overline{z}^{-a_{k}-\frac{1}{2}})_{1\leq k\leq n-1}$. Hypothesis \ref{hyp} becomes:
\begin{hyp}\label{hypdoubleodd}
$$-a_{n}>a_{1}'\geq -a_{n-1}>a_{2}' \geq -a_{n-3}>a_{3}' \cdots \geq -a_{2}>a_{n-1}'\geq -a_{1}$$
\end{hyp}
Under this hypothesis, we can apply Theorem \ref{maintheorem2} to $(\Pi,\Pi')$. We get, for critical point $s=m+\frac{1}{2}$ of $L(s,\Pi\times \Pi')$ with $m\geq 0$,
$$L(m+\frac{1}{2},\Pi\times \Pi')\sim_{E(\Pi)E(\Pi')} p(m,\Pi_{\infty},\Pi'_{\infty})Z(\Pi_{\infty})Z(\Pi'_{\infty})\prod\limits_{j=1}^{n-1}P^{(j)}(\Pi)\prod\limits_{k=1}^{n-2}P^{(k)}(\Pi')$$
equivariant under action of $G_{K}$.

By equation (\ref{oddfinal}) and (\ref{evenfinal}), for all $1\leq j\leq n-1$ and $1\leq k\leq n-2$
$$P^{(j)}(\Pi)\sim_{E(\chi)} D_{F^{+}}^{1/2}\mathcal{G}(\varepsilon_{K})^{-\frac{n-1}{2}}p(\check{\chi},\Phi_{j,\chi});$$
$$P^{(k)}(\Pi') \sim_{E(\chi')E(\psi)} D_{F'^{+}}^{1/2}(2\pi i)^{-\frac{n-1}{2}}\mathcal{G}(\varepsilon_{K})^{-\frac{n-1}{2}}p(\check{\chi}',\Phi_{k,\chi'})p(\psi)^{k}p(\psi^{c})^{n-1-k}$$
equivariant under action of $G_{K}$.

Note that $\Phi_{j,\chi}=\{\bar{\sigma}_{1},\bar{\sigma}_{2},\cdots,\bar{\sigma}_{n-j},\sigma_{n-j+1},\cdots,\sigma_{n-1},\sigma_{n}\}$. We have
$$p(\check{\chi},\Phi_{j,\chi})\sim_{E(\chi)}\prod\limits_{1\leq i\leq n-j}p(\check{\chi},\{\bar{\sigma}_{i}\})\prod\limits_{n-j+1\leq i\leq n}p(\check{\chi},\{\sigma_{i}\})$$
equivariant under action of $G_{K}$ by equation (\ref{separateCMtype}) in Proposition \ref{propCM}.

Similarly, we have that equivariant under action of $G_{K}$: $$p(\check{\chi}',\Phi_{k,\chi'})\sim_{E(\chi')}\prod\limits_{1\leq i\leq n-k-1}p(\check{\chi}',\{\bar{\sigma'}_{i}\})\prod\limits_{n-k\leq i\leq n-1}p(\check{\chi}',\{\sigma'_{i}\}).$$

Therefore, we deduce that for critical point $s=m+n-1$ of $M(\Pi)\otimes M(\Pi')$ with $m\geq 0$,
\begin{eqnarray}\label{double1}
&L(m+\frac{1}{2},\Pi\times \Pi')\sim_{E(\chi)E(\chi')E(\psi)} &\nonumber\\
& D_{F^{+}}^{\frac{n-1}{2}}D_{F'^{+}}^{\frac{n-2}{2}}\mathcal{G}(\varepsilon_{K})^{-\frac{(2n-3)(n-1)}{2}}p(m,\Pi_{\infty},\Pi'_{\infty})Z(\Pi_{\infty})Z(\Pi'_{\infty})\times&\nonumber\\
&(2\pi i)^{-\frac{(n-1)(n-2)}{2}}[p(\psi)p(\psi^{c})]^{\frac{(n-1)(n-2)}{2}}\prod\limits_{j=1}^{n}p(\check{\chi},\sigma_{j})^{j-1}p(\check{\chi},\bar{\sigma}_{j})^{n-j}
\prod\limits_{k=1}^{n-1}p(\check{\chi}',\sigma_{k})^{k-1}p(\check{\chi}',\bar{\sigma'_{k}})^{n-1-k}& \nonumber\\
&&
\end{eqnarray}
equivariant under action of $G_{K}$.

Since $\psi\psi^{c}=||\cdot||_{\AK}$, we have
$$p(\psi)p(\psi^{c})=p(\check{\psi},1)p(\check{\psi^{c}},1)\sim_{E(\psi)} p(||\cdot||^{-1}_{\AK},1)\sim _{E(\psi)}2\pi i$$
equivariant under action of $G_{K}$ by equation (\ref{charmulti}) and (\ref{norm}).

Thus, equation (\ref{double1}) becomes
\begin{eqnarray}\label{double2}
&L(m+\frac{1}{2},\Pi\times \Pi')&\sim_{E(\chi)E(\chi')E(\psi)}
D_{F^{+}}^{\frac{n-1}{2}}D_{F'^{+}}^{\frac{n-2}{2}}\mathcal{G}(\varepsilon_{K})^{-\frac{(2n-3)(n-1)}{2}}p(m,\Pi_{\infty},\Pi'_{\infty})Z(\Pi_{\infty})Z(\Pi'_{\infty})\nonumber\\
 &&\times\prod\limits_{j=1}^{n}p(\check{\chi},\sigma_{j})^{j-1}p(\check{\chi},\bar{\sigma}_{j})^{n-j}
\prod\limits_{k=1}^{n-1}p(\check{\chi}',\sigma_{k})^{k-1}p(\check{\chi}',\bar{\sigma'_{k}})^{n-1-k}.
\end{eqnarray}
equivariant under action of $G_{K}$.

On the other hand, we have
\begin{eqnarray}
& &L(m+\frac{1}{2},\Pi\times \Pi')\nonumber \\
&=& L(m+\frac{1}{2}, \Pi(\chi)\times  \Pi(\chi')\otimes ||\cdot||_{\AK}^{-\frac{1}{2}} \psi) \nonumber\\
&=& L(m, \Pi(\chi)\times  \Pi(\chi')\otimes\psi) \nonumber\\
&=& L(m,(\chi\circ N_{\AL/\AF})(\chi'\circ N_{\AL/\AFF}) (\psi\circ N_{\AL/\AK})).
\end{eqnarray}

The infinity type of $(\chi\circ N_{\AL/\AF})(\chi'\circ N_{\AL/\AFF}) (\psi\circ N_{\AL/\AK})$ is $(z^{a_{j}+a_{k}'+1}\overline{z}^{-a_{j}-a_{k}'})_{1\leq j\leq n,1\leq k\leq n-1}$. It is always critical, i.e. $a_{j}+a_{k}'+1\neq -a_{j}-a_{k}'$ for all $j,k$. Hence there exists $\Phi_{\chi,\chi'}$ a unique CM type of $L$ compatible with $(\chi\circ N_{\AL/\AF})(\chi'\circ N_{\AL/\AFF}) (\psi\circ N_{\AL/\AK})$. More precisely, $\tau_{j,k}\in \Phi_{\chi,\chi'}$ if and only if $a_{j}+a_{k}'<0$. By Hypothesis \ref{hypdoubleodd}, it is equivalent to the condition that $j+k\geq n+1$. similarly, $\overline{\tau}_{j,k}\in \Phi_{\chi,\chi'}$ if and only if $j+k\leq n$.

Therefore, by Blasius' result,
\begin{eqnarray}\label{blasiusdouble}
& &L(m+\frac{1}{2},\Pi\times \Pi')\nonumber \\
& \sim_{E(\chi)E(\chi')E(\psi)}& D_{F^{+}F'^{+}}^{1/2}(2\pi i)^{mn(n-1)} p((\check{\chi}\circ N_{\AL/\AF})(\check{\chi}'\circ N_{\AL/\AFF}) (\check{\psi}\circ N_{\AL/\AK}),\Phi_{\chi,\chi'})\nonumber \\
& \sim_{E(\chi)E(\chi')E(\psi)}& (D_{F^{+}})^{\frac{n-1}{2}}(D_{F'^{+}})^{\frac{n}{2}}(2\pi i)^{mn(n-1)} p((\check{\chi}\circ N_{\AL/\AF})(\check{\chi}'\circ N_{\AL/\AFF}) (\check{\psi}\circ N_{\AL/\AK}),\Phi_{\chi,\chi'})\nonumber\\
&&
\end{eqnarray}
equivariant under action of $G_{K}$.
The last equation is due to $D_{F^{+}F'^{+}}^{1/2}\sim_{\Q} (D_{F^{+}})^{\frac{n-1}{2}}(D_{F'^{+}})^{\frac{n}{2}}$ by Lemma \ref{lemmadiscriminant}.

Moreover, by Proposition \ref{propCM},
\begin{eqnarray}
& &p((\check{\chi}\circ N_{\AL/\AF})(\check{\chi}'\circ N_{\AL/\AFF}) (\check{\psi}\circ N_{\AL/\AK}),\Phi_{\chi,\chi'}) \nonumber \\
&\sim_{E(\chi)E(\chi')E(\psi)}&\prod\limits_{j+k\geq n+1} p((\check{\chi}\circ N_{\AL/\AF})(\check{\chi}'\circ N_{\AL/\AFF}) (\check{\psi}\circ N_{\AL/\AK}),\tau_{j,k})\times \nonumber\\
& &\prod\limits_{j+k\leq n} p((\check{\chi}\circ N_{\AL/\AF})(\check{\chi}'\circ N_{\AL/\AFF}) (\check{\psi}\circ N_{\AL/\AK}),\overline{\tau}_{j,k})\nonumber \\
&\sim_{E(\chi)E(\chi')E(\psi)}&  \prod\limits_{j+k\geq n+1} p(\check{\chi},\tau_{j,k}|_{\AF})p(\check{\chi}',\tau_{j,k}|_{\AFF})p(\check{\psi},\tau_{j,k}|_{\AK})\times \nonumber\\
& &\prod\limits_{j+k\leq n} p(\check{\chi},\overline{\tau}_{j,k}|_{\AF})p(\check{\chi}',\overline{\tau}_{j,k}|_{\AFF})p(\check{\psi},\overline{\tau}_{j,k}|_{\AK})\nonumber \\
&\sim_{E(\chi)E(\chi')E(\psi)}&  \prod\limits_{j+k\geq n+1} p(\check{\chi},\sigma_{j})p(\check{\chi}',\sigma_{k}')p(\check{\psi},1)\times\prod\limits_{j+k\leq n} p(\check{\chi},\overline{\sigma}_{j})p(\check{\chi}',\overline{\sigma}_{k})p(\check{\psi},\iota)\nonumber\\
&\sim_{E(\chi)E(\chi')E(\psi)}& \prod\limits_{j=1}^{n}p(\check{\chi},\sigma_{j})^{j-1}p(\check{\chi}',\overline{\sigma}_{j})^{n-j} \times\prod\limits_{k=1}^{n-1}p(\check{\chi}',\sigma_{k}')^{k}p(\check{\chi}',\overline{\sigma}_{k}')^{n-k}\times\nonumber\\
&&[p(\check{\psi},1)p(\check{\psi},\iota)]^{\frac{n(n-1)}{2}}\nonumber\\
&\sim_{E(\chi)E(\chi')E(\psi)}& (2\pi i)^{\frac{n(n-1)}{2}} \prod\limits_{j=1}^{n}p(\check{\chi},\sigma_{j})^{j-1}p(\check{\chi}',\overline{\sigma}_{j})^{n-j} \times\prod\limits_{k=1}^{n-1}p(\check{\chi}',\sigma_{k}')^{k}p(\check{\chi}',\overline{\sigma}_{k}')^{n-k}\nonumber
\end{eqnarray}
equivariant under action of $G_{K}$.

We assume moreover $L(m+\frac{1}{2},\Pi\times \Pi')\neq 0$. It is always true when $m\geq 1$ since in this case $m$ is in the absolutely convergent range. We compare the above equation with equation (\ref{double2}) and (\ref{blasiusdouble}) and get:
$$p(m,\Pi_{\infty},\Pi'_{\infty})Z(\Pi_{\infty})Z(\Pi'_{\infty})$$
$$\sim _{E(\chi)E(\chi')E(\psi)} (2\pi i)^{(m+\frac{1}{2})n(n-1)}\mathcal{G}(\varepsilon_{K})^\frac{(2n-3)(n-1)}{2}\prod\limits_{k=1}^{n-1}p(\check{\chi}',\sigma_{k}')p(\check{\chi}',\overline{\sigma}_{k}')$$
equivariant under action of $G_{K}$.

Note that $\chi'$ is conjugate self-dual, thus $\chi'\chi'^{c}$ is the trivial character and then we have
$$
p(\check{\chi}',\sigma_{k}')p(\check{\chi}',\overline{\sigma}_{k}')\sim_{E(\chi')} p(\check{\chi}',\sigma_{k}')p(\check{\chi}'^{c},\sigma_{k}')\sim_{E(\chi')} p(\check{\chi}'\check{\chi'^{c}},\sigma_{k}')\sim_{E(\chi')} 1
$$
equivariant under action of $G_{K}$.

We remark that $\mathcal{G}(\varepsilon_{K})\in K^{\times}$ and hence $\mathcal{G}(\varepsilon_{K})\sim_{K} 1$ equivariant under action of $G_{K}$. Hence we have $$p(m,\Pi_{\infty},\Pi'_{\infty})Z(\Pi_{\infty})Z(\Pi'_{\infty})\sim _{E(\chi)E(\chi')E(\psi)}
(2\pi i)^{(m+\frac{1}{2})n(n-1)}$$
equivariant under action of $G_{K}$.

Note that both sides of the above equation depend only on $\chi_{\infty}$, $\chi'_{\infty}$ and $\psi_{\infty}=z$, thus the relation $\sim _{E(\chi)E(\chi')E(\psi)}$ can be improved to $\sim _{KE(\chi_{\infty})E(\chi'_{\infty})}$.
\begin{rem}
We have assumed both $E(\chi)$ and $E(\chi')$ contain $K$. But $E(\chi_{\infty})$ or $E(\chi'_{\infty})$ defined as the definition fields for certain algebraic representations in Section \ref{notation} may not contain $K$ in general.
\end{rem}

We deduce finally that:
\begin{thm}\label{applicationodd}
Assume $n$ is an odd integer. Let $K$ be a quadratic field and $F=F^{+}K$, $F'=F'^{+}K$ be two CM fields of degree $n$ and $n-1$ over $K$ and cyclic over $K$. Let $\chi$ and $\chi'$ be two algebraic, conjugate self-dual and very regular Hecke character of $F$ and $F'$ satisfying Hypothesis \ref{hypsupercuspidal}.

We write $\Pi:=\Pi(\chi)$ and $\Pi':=\Pi(\chi')\otimes ||\cdot||_{\AK}^{-\frac{1}{2}}\psi$ where $\psi$ is a Hecke character of $K$ such that $\psi\psi^{c}=||\cdot ||_{\AK}$ and $\psi_{\infty}(z)=z$. 

Let $m\geq 0$ be an integer such that $m+n-1$ is critical for $M(\Pi)\otimes M(\Pi')$. If $m=0$, we assume moreover that $L(m+\frac{1}{2},\Pi\times \Pi')\neq 0$. 

We then have:
$$p(m,\Pi_{\infty},\Pi'_{\infty})Z(\Pi_{\infty})Z(\Pi'_{\infty})\sim _{KE(\chi_{\infty})E(\chi'_{\infty})} (2\pi i)^{(m+\frac{1}{2})n(n-1)}$$
equivariant under action of $G_{K}$.

Here $p(m,\Pi_{\infty},\Pi'_{\infty})$, $Z(\Pi_{\infty})$ and $Z(\Pi'_{\infty})$ are Archimedean local factors defined in chapter $6$ of \cite{harrismotivic}.
\end{thm}

\end{subsection}

\begin{subsection}{The case $n$ even}

Now we assume $n$ to be even. In this case, the infinity type of $\Pi$ is $(z^{a_{j}+\frac{1}{2}}\overline{z}^{-a_{j}-\frac{1}{2}})_{1\leq j\leq n}$ and the infinity type of $\Pi'$ is $(z^{a_{k}}\overline{z}^{-a_{k}})_{1\leq k\leq n-1}$. The Hypothesis \ref{hyp} remains the same as \ref{hypdoubleodd}
$$-a_{n}>a_{1}'\geq -a_{n-1}>a_{2}' \geq -a_{n-3}>a_{3}' \cdots \geq -a_{2}>a_{n-1}'\geq -a_{1}$$

Under this hypothesis, for $s=m+n-1$ a critical value of $M(\Pi)\otimes M(\Pi')$ with $m\geq 0$,
$$L(m+\frac{1}{2},\Pi\times \Pi')\sim_{E(\Pi)E(\Pi')} p(m,\Pi_{\infty},\Pi'_{\infty})Z(\Pi_{\infty})Z(\Pi'_{\infty})\prod\limits_{j=1}^{n-1}P^{(j)}(\Pi)\prod\limits_{k=1}^{n-2}P^{(k)}(\Pi')$$
equivariant under action of $G_{K}$.

By equation (\ref{oddfinal}) and (\ref{evenfinal}), for all $1\leq j\leq n-1$ and $1\leq k\leq n-2$
$$P^{(j)}(\Pi)\sim_{E(\chi)} D_{F^{+}}^{\frac{1}{2}}(2\pi i)^{-\frac{n}{2}}\mathcal{G}(\varepsilon_{K})^{-\frac{n}{2}}p(\check{\chi},\Phi_{j,\chi})p(\psi)^{j}p(\psi^{c})^{n-j};$$
$$P^{(k)}(\Pi') \sim_{E(\chi')E(\psi)} D_{F'^{+}}^{\frac{1}{2}}\mathcal{G}(\varepsilon_{K})^{-\frac{n-2}{2}}p(\check{\chi}',\Phi_{k,\chi'})$$
equivariant under action of $G_{K}$.

We repeat the calculation in previous section and get, for a critical point $s=m+\frac{1}{2}$ of $L(s,\Pi\times \Pi')$ with $m\geq 0$:


\begin{eqnarray}\label{double2even}
&L(m+\frac{1}{2},\Pi\times \Pi')&\sim_{E(\chi)E(\chi')E(\psi)}
D_{F^{+}}^{\frac{n-1}{2}}D_{F'^{+}}^{\frac{n-2}{2}}\mathcal{G}(\varepsilon_{K})^{-\frac{2n^2-5n+4}{2}}p(m,\Pi_{\infty},\Pi'_{\infty})Z(\Pi_{\infty})Z(\Pi'_{\infty})\times\nonumber\\
 &&\prod\limits_{j=1}^{n}p(\check{\chi},\sigma_{j})^{j-1}p(\check{\chi},\bar{\sigma}_{j})^{n-j}
\prod\limits_{k=1}^{n-1}p(\check{\chi}',\sigma_{k})^{k-1}p(\check{\chi}',\bar{\sigma'_{k}})^{n-1-k}
\end{eqnarray}
equivariant under action of $G_{K}$.
\bigskip

On the other hand,
\begin{eqnarray}
& &L(m+\frac{1}{2},\Pi\times \Pi')\nonumber \\
&=& L(m+\frac{1}{2}, \Pi(\chi)\otimes ||\cdot||_{\AK}^{-\frac{1}{2}} \psi\times  \Pi(\chi')) \nonumber\\
&=& L(m, \Pi(\chi)\otimes\psi\times  \Pi(\chi')) \nonumber\\
&=& L(m,(\chi\circ N_{\AL/\AF})(\chi'\circ N_{\AL/\AFF}) (\psi\circ N_{\AL/\AK}))\nonumber.
\end{eqnarray}

As in the previous section, by Blasius' result and Proposition \ref{propCM}, we have that equivariant under action of $G_{K}$:
\begin{eqnarray}\label{blasiusdoubleeven}
&& L(m+\frac{1}{2},\Pi\times \Pi')\nonumber\\
&\sim_{E(\chi)E(\chi')E(\psi)}&D_{F^{+}F'^{+}}^{\frac{1}{2}}(2\pi i)^{mn(n-1)} p((\check{\chi}\circ N_{\AL/\AF})(\check{\chi}'\circ N_{\AL/\AFF}) (\check{\psi}\circ N_{\AL/\AK}),\Phi_{\chi,\chi'})\nonumber\\
&\sim_{E(\chi)E(\chi')E(\psi)}& D_{F^{+}F'^{+}}^{\frac{1}{2}}(2\pi i)^{mn(n-1)}p((\check{\chi}\circ N_{\AL/\AF})(\check{\chi}'\circ N_{\AL/\AFF}) (\check{\psi}\circ N_{\AL/\AK}),\Phi_{\chi,\chi'}) \nonumber \\
&\sim_{E(\chi)E(\chi')E(\psi)}& D_{F^{+}F'^{+}}^{\frac{1}{2}}(2\pi i)^{(m+\frac{1}{2})n(n-1)} \prod\limits_{j=1}^{n}p(\check{\chi},\sigma_{j})^{j-1}p(\check{\chi}',\overline{\sigma}_{j})^{n-j} \times\prod\limits_{k=1}^{n-1}p(\check{\chi}',\sigma_{k}')^{k}p(\check{\chi}',\overline{\sigma}_{k}')^{n-k}\nonumber\\
&\sim_{E(\chi)E(\chi')E(\psi)}&D_{F^{+}}^{\frac{n-1}{2}}D_{F'^{+}}^{\frac{n}{2}} (2\pi i)^{(m+\frac{1}{2})n(n-1)} \prod\limits_{j=1}^{n}p(\check{\chi},\sigma_{j})^{j-1}p(\check{\chi}',\overline{\sigma}_{j})^{n-j} \times\prod\limits_{k=1}^{n-1}p(\check{\chi}',\sigma_{k-1}')^{k}p(\check{\chi}',\overline{\sigma}_{k}')^{n-k-1}\nonumber\\
\text{}
\end{eqnarray}
where the last equation is due to the fact that $\chi\chi^{c}$ is trivial.

If $L(m+\frac{1}{2},\Pi\times \Pi')\neq 0$,  we compare the above equation with equation (\ref{double2even}) and (\ref{blasiusdoubleeven}) and get:
$$p(m,\Pi_{\infty},\Pi'_{\infty})Z(\Pi_{\infty})Z(\Pi'_{\infty})\sim _{KE(\chi)E(\chi')E(\psi)}
(2\pi i)^{(m+\frac{1}{2})n(n-1)}\mathcal{G}(\varepsilon_{K})^\frac{2n^2-5n+4}{2}$$
equivariant under action of $G_{K}$.

Finally, since $\mathcal{G}(\varepsilon_{K})\sim_{K} 1$ equivariant under action of $G_{K}$ and both sides of the above equation depend only on $\chi_{\infty}$, $\chi'_{\infty}$ and $\psi_{\infty}=z$, we deduce that:
\begin{thm}\label{applicationeven}
Assume $n$ is an even integer. Let $K $ be a quadratic field and $F=F^{+}K$, $F'=F'^{+}K$ be two CM field of degree $n$ and $n-1$ over $K$ and cyclic over $K$. Let $\chi$ and $\chi'$ be two algebraic, conjugate self-dual and very regular Hecke character of $F$ and $F'$ satisfying Hypothesis \ref{hypsupercuspidal}.

We write $\Pi:=\Pi(\chi)\otimes ||\cdot||_{\AK}^{-\frac{1}{2}}\psi$ and $\Pi'=\Pi(\chi')$ where $\psi$ is a Hecke character of $K$ such that $\psi\psi^{c}=||\cdot||_{\AK}$ and $\psi_{\infty}(z)=z$. 

Let $m\geq 0$ be an integer such that $m+n-1$ is critical for $M(\Pi)\otimes M(\Pi')$. If $m=0$, we assume moreover that $L(m+\frac{1}{2},\Pi\times \Pi')\neq 0$.  

We have:
$$p(m,\Pi_{\infty},\Pi'_{\infty})Z(\Pi_{\infty})Z(\Pi'_{\infty})\sim_{KE(\chi_{\infty})E(\chi'_{\infty})}
(2\pi i)^{(m+\frac{1}{2})n(n-1)}$$
equivariant under action of $G_{K}$.

Here $p(m,\Pi_{\infty},\Pi'_{\infty})$, $Z(\Pi_{\infty})$ and $Z(\Pi'_{\infty})$ are Archimedean local factors defined in chapter $6$ of \cite{harrismotivic}.
\end{thm}

\end{subsection}

\begin{subsection}{Final remarks}
In both cases ($n$ even or odd), we have got the same result:
$$p(m,\Pi_{\infty},\Pi'_{\infty})Z(\Pi_{\infty})Z(\Pi'_{\infty})\sim_{KE(\Pi_{\infty})E(\Pi'_{\infty})}
(2\pi i)^{(m+\frac{1}{2})n(n-1)}$$
here $\Pi_{\infty}$ and $\Pi'_{\infty}$ are obtained from Hecke characters as in previous sections.

Note that for any $(\Pi,\Pi')$ as in Section \ref{ntimesn-1}, we may find Hecke characters $\chi$ and $\chi'$ such that $\Pi_{\infty}$ and $\Pi'_{\infty}$ are the same as the representations we got from $\chi$ and $\chi'$ in previous sections. Therefore, the above result can be generalized to any pair $(\Pi,\Pi')$ which satisfy the conditions in Section \ref{ntimesn-1}. Let us give more details on this point.

\begin{df}
Let $K$ be a quadratic imaginary field and $F$ be a CM field containing $K$. For $\chi$ an algebraic Hecke character of $K$, define $\Pi_{\chi}:=\Pi(\chi)$ if the degree of $F$ over $K$ is odd; $\Pi_{\chi}:=\Pi(\chi)\otimes ||\cdot||_{\AK}^{-\frac{1}{2}}\psi$ if the degree of $F$ over $K$ is even where $\Pi(\chi)$ is the automorphic induction of $\chi$ from $GL_{1}(\AF)$ to $GL_{n}(\AK)$ where $n$ is the degree of $F$ over $K$ and $\psi$ is a fixed algebraic Hecke character of $K$ with infinity type $z^{1}\overline{z}^{0}$ such that $\psi\psi^{c}=||\cdot||_{\AK}$ (c.f. Lemma \ref{psi}).
\end{df}

We first restate Theorem \ref{applicationodd} and Theorem \ref{applicationeven} together:
\begin{thm}
Let $K$ be a quadratic imaginary field and $F^{+}$ (resp. $F'^{+}$) be a totally real field of degree $n$ (resp. $n-1$) over $\Q$.  Let $F=F^{+}K$ and $F'=F'^{+}K$. Assume $F$ and $F'$ are cyclic over $K$.

Let $\chi$ and $\chi'$ be two algebraic regular conjugate self-dual Hecke characters of $F$ and $F'$ satisfying Hypothesis \ref{hypsupercuspidal} and Hypothesis \ref{hypdoubleodd}. Put $\Pi=\Pi_{\chi}$ and $\Pi'=\Pi_{\chi'}$. 

Let $m\geq 0$ be an integer such that $m+n-1$ is critical for $M(\Pi)\otimes M(\Pi')$. If $m=0$, we assume moreover that $L(m+\frac{1}{2},\Pi\times \Pi')\neq 0$. 

Then the archimedean local factors defined in chapter $6$ of \cite{harrismotivic} satisfy that:

$$p(m,\Pi_{\infty},\Pi'_{\infty})Z(\Pi_{\infty})Z(\Pi'_{\infty})\sim _{KE(\chi_{\infty})E(\chi'_{\infty})}
(2\pi i)^{(m+\frac{1}{2})n(n-1)}.$$
\end{thm}

\begin{lem}\label{chiandchi'}
Let $n$ be an integer positive and $K$ be a quadratic imaginary field. Let $F^{+}$ be a totally real field of degree $n$ over $\Q$. Put $F=F^{+}K$ a CM field.
If $\Pi$ is an algebraic cuspidal representation of $GL_{n}(K)$ then there exists $\chi$ an algebraic Hecke character of $F$ which satisfies Hypothesis \ref{hypsupercuspidal} such that $\Pi_{\infty}\cong \Pi_{\chi,\infty}$.

Furthermore, if $\Pi$ is conjugate self-dual, we may have in addition that $\chi$ is conjugate self-dual.
\end{lem}

\begin{dem}
We denote the infinity type of $\Pi$ by $(z^{a_{i}}\overline{z}^{b_{i}})_{1\leq i\leq n}$ with $a_{i}$, $b_{i}\in \Z+\cfrac{n-1}{2}$. By Lemma \ref{wPi}, $a_{i}+b_{i}$ is a constant independent of $i$. If $n$ is odd, we may take $\chi$ an algebraic Hecke character of $F$ satisfying Hypothesis \ref{hypsupercuspidal} with infinity type ($\sigma_{i}^{a_{i}}\overline{\sigma}_{i}^{b_{i}})_{1\leq i\leq n}$. If $n$ is even, may take $\chi$ an algebraic Hecke character of $F$ satisfying Hypothesis \ref{hypsupercuspidal} with infinity type $(\sigma_{i}^{a_{i}-\frac{1}{2}}\overline{\sigma}_{i}^{b_{i}+\frac{1}{2}})_{1\leq i\leq n}$. Here $\sigma_{1},\cdots, \sigma_{n}$ are the embeddings from $F$ to $\C$ which are trivial on $K$ as before. The existence of $\chi$ is guaranteed by the fact that $a_{i}+b_{i}$ is a constant (see Lemma $4.1.1$ and paragraphs before Lemma $4.1.3$ in \cite{CHT}). It is easy to see that $\Pi_{\infty}$ and $\Pi_{\chi,\infty}$ have the same infinity type and then are isomorphic to each other.

If $\Pi$ is conjugate self-dual, we see that $\chi_{\infty}$ is conjugate self-dual. In particular, $\chi|_{F^{+}}$ is trivial at infinity places. By Lemma $4.1.4$ of \cite{CHT}, we may find $\phi$ an algebraic Hecke character of $F$ with trivial infinity type such that $\phi\phi^{c}=\chi\chi^{c}$. From the proof of Lemma $4.1.4$ of \textit{loc.cit}, we see that $\psi$ can be trivial at places over $q$. Put $\chi'=\chi\phi^{-1}$. It is then a conjugate self-dual Hecke character which satisfies Hypothesis \ref{hypsupercuspidal} and the condition $\Pi_{\infty}\cong \Pi_{\chi',\infty}$.

\end{dem}

The above lemma allow us to generalize Theorem \ref{applicationodd} and Theorem \ref{applicationeven} to arbitrary pair $(\Pi,\Pi')$ which satisfies the conditions in Theorem \ref{maintheorem2}. We remark that Hypothesis \ref{hyp} and Hypothesis \ref{hypregular} only concern the infinity type of the representation. Therefore, if $(\Pi,\Pi')$ satisfies Hypothesis \ref{hyp} and \ref{hypregular}, then $(\Pi_{\chi},\Pi_{\chi'})$ also satisfies these hypothesis where $\chi$ and $\chi'$ are characters associated to $\Pi$ and $\Pi'$ as in the above lemme. 

Note that an extra condition on the non vanishing of $L$-function will be needed when $m=0$:

\begin{hyp}\label{m=0}
For $\Pi$ and $\Pi'$ conjugate self-dual algebraic cuspidal representations of $GL_{n}(\AK)$ and $GL_{n-1}(\AK)$, there exists Hecke characters $\chi$ and $\chi'$ of $F$ and $F'$ such that:
\begin{itemize}
\item $\chi$ and $\chi'$ are conjugate self-dual;
\item $\Pi_{\infty}\cong \Pi_{\chi,\infty}$ and $\Pi_{\infty}'\cong \Pi_{\chi',\infty}$;
\item $L(\cfrac{1}{2}, \Pi_{\chi}\times \Pi_{\chi'})\neq 0$.
\end{itemize}
\end{hyp}

We can generalize Theorem \ref{applicationodd} and Theorem \ref{applicationeven} as follows:

\begin{thm}
Let $\Pi$ and $\Pi'$ be cuspidal representations of $GL_{n}(\AK)$ and $GL_{n-1}(\AK)$ as in Section \ref{ntimesn-1}, i.e. $\Pi$ and $\Pi'$ are very regular, cohomological, conjugate self-dual, supercuspidal at places over $q$ and their infinity types satisfy Hypothesis \ref{hyp}. 

Let $m\geq 0$ be an integer such that $m+n-1$ is critical for $M(\Pi)\otimes M(\Pi')$. If $m=0$, we assume moreover Hypothesis \ref{m=0}. 

We then have $$p(m,\Pi_{\infty},\Pi'_{\infty})Z(\Pi_{\infty})Z(\Pi'_{\infty})\sim _{KE(\Pi_{\infty})E(\Pi'_{\infty})}
(2\pi i)^{(m+\frac{1}{2})n(n-1)}$$
equivariant under action of $G_{K}$.

Consequently, 
$$L(m+\frac{1}{2},\Pi\times \Pi')\sim_{E(\Pi)E(\Pi')} (2\pi i)^{(m+\frac{1}{2})n(n-1)}\prod\limits_{j=1}^{n-1}P^{(j)}(\Pi)\prod\limits_{k=1}^{n-2}P^{(k)}(\Pi')$$
equivariant under action of $G_{K}$. 
\end{thm}

\begin{rem}
The above result is compatible with the Deligne conjecture and M. Harris's calculation on the Deligne period.

Recall that the Deligne conjecture predicts
$$L(n-1+m, M(\Pi)\otimes M(\Pi'))\sim c^{+}(M(\Pi)\otimes M(\Pi')(n-1+m)).$$

The equation $(4.12)$ of \cite{harrismotivic} gives
$$c^{+}(M(\Pi)\otimes M(\Pi')(n-1+m)) \sim (2\pi i)^{(m+\frac{1}{2})n(n-1)} \prod\limits_{j=1}^{n-1}P_{\leq j}(\Pi)\prod\limits_{k=1}^{n-2}P_{\leq k}(\Pi')$$ (see chapter $4$ of \cite{harrismotivic} for the notion).
From the discussion after Theorem $4.27$ in \cite{harrismotivic} we see that $P^{(s)}\sim P_{\leq s}$ in our case.

\end{rem}

\end{subsection}

\end{section}

\bibliography{bibfile}

\begin{thebibliography}{CHT08}

\bibitem[AC89]{arthurclozel}
J.~Arthur and L.~Clozel.
\newblock {\em Simple algebras, base change, and the advanced theory of the
  trace formula}.
\newblock Number 120 in Annals of Mathematics Studies. Princeton University
  Press, 1989.

\bibitem[Art03]{arthurtraceformula}
J.~Arthur.
\newblock An introduction to the trace formula.
\newblock In J.~Arthur, D.~Ellwood, and R.~Kottwitz, editors, {\em Harmonic
  analysis, the trace formula and {S}himura varieties}, volume~4 of {\em Clay
  Mathematics Proceedings}, pages 1--264. American Mathematical Society Clay
  Mathematics Institute, 2003.

\bibitem[CHT08]{CHT}
L.~Clozel, M.~Harris, and R.~Taylor.
\newblock Automorphy for some $l$-adic lifts of automorphic mod $l$ {G}alois
  representations.
\newblock {\em Publications math{\'e}matiques de l'I.H.E.S}, 108:1--181, 2008.

\bibitem[CL99]{clozellabesse}
L.~Clozel and J.~P. Labesse.
\newblock Changement de base pour les repr{\'e}sentations cohomologiques de
  certains groupes unitaires.
\newblock {\em Ast{\'e}risque}, (257):121--136, 1999.

\bibitem[Clo90]{clozelaa}
L.~Clozel.
\newblock Motifs et formes automorphes: applications du principe de
  fonctorialit{\'e}.
\newblock In {\em Automorphic forms, $S$himura varieties, and $L$-functions},
  volume~1 of {\em Perspectives in Mathematics}, pages 77--159. Boston, MA:
  Academic Press, 1990.

\bibitem[Clo91]{clozelIHES}
L.~Clozel.
\newblock Rerpr{\'e}sentations galoisiennes associ{\'e}es aux
  repr{\'e}sentations automorphes autoduales de {G}l(n).
\newblock {\em Publications math{\'e}matiques de l'I.H.E.S}, 73:97--145, 1991.

\bibitem[Del79]{deligne79}
P.~Deligne.
\newblock Valeurs de fonctions {L} et p\'{e}riodes d'int\'{e}grales.
\newblock In A.~Borel and W.~Casselman, editors, {\em Automorphic forms,
  representations and {L}-functions}, volume~33 of {\em Proceedings of the
  symposium in pure mathematics of the {A}merican mathematical society}.
  American Mathematical Society, 1979.

\bibitem[GH]{harrismotivic}
H.~Grobner and M.~Harris.
\newblock Whittaker periods, motivic periods, and special values of tensor
  product of {L}-functions.

\bibitem[Har93]{harrisCMperiod}
M.~Harris.
\newblock {L}-functions of 2$\times$2 unitary groups and factorization of
  periods of {H}ilbert modular forms.
\newblock {\em J. Am. Math. Soc}, 6(3):637--719, 1993.

\bibitem[Har97]{harris97}
M.~Harris.
\newblock {L}-functions and periods of polarized regular motives.
\newblock {\em J. Reine Angew. Math}, (483):75--161, 1997.

\bibitem[Har98]{harrislanglands}
M.~Harris.
\newblock The local langlands conjecture for ${G}l_{n}$ over a $p$-adic field,
  n<p.
\newblock {\em Invent. math.}, (134):177--210, 1998.

\bibitem[Har07]{harrisunitaryperiod}
M.~Harris.
\newblock Cohomological automorphic forms on unitary groups, {II}: Period
  relations and values of {L}-functions.
\newblock In {\em Harmonic analysis, group representations, automorphic forms
  and invariant theory}, volume~12 of {\em Lecture Notes Series, Institute of
  Mathematical Sciences, National University of Singapore (volume in honor of
  Roger Howe)}. World Scientific Publishing, 2007.

\bibitem[Har08]{harrissimpleproof}
M.~Harris.
\newblock A simple proof of rationality of {S}iegel-{W}eil {E}isenstein series.
\newblock In W.~T. Gan, S.~S. Kudla, and Y.~Tschinkel, editors, {\em Eisenstein
  Series and Applications}, volume 258 of {\em Birkh{\"a}user: Progress in
  Mathematics}, pages 149--186, 2008.

\bibitem[Har10]{harristakagi}
M.~Harris.
\newblock Arithmetic applications of the langlands program.
\newblock {\em Japanese Journal of Mathematics}, (5):1--71, 2010.

\bibitem[Har13]{harrisimrn}
M.~Harris.
\newblock Beilinson-bernstein localization over $\mathcal{Q}$ and periods of
  automorphic forms.
\newblock {\em International Math. Research Notices}, pages 2000--2053, 2013.

\bibitem[HK91]{harrisappendix}
M.~Harris and S.~S. Kudla.
\newblock The central critical value of the triple product {L}-functions.
\newblock {\em Annals of {M}athematics, Second Series}, 133(3):605--672, 1991.

\bibitem[HL04]{harrislabesse}
M.~Harris and J.~P. Labesse.
\newblock Conditional base change for unitary groups.
\newblock {\em Asian J. Math.}, 8(4):653--684, 2004.

\bibitem[HT01]{harristaylor}
M.~Harris and R.~Taylor.
\newblock {\em The geometry and cohomology of some simple {S}himura varieties}.
\newblock Number 151 in Annals of Mathematics Studies. Princeton University
  Press, 2001.

\bibitem[Min11]{minguez}
A.~Minguez.
\newblock Unramified representations of unitary groups.
\newblock In L.~Clozel, M.~Harris, J.~P. Labesse, and B.~C. Ng{\^o}, editors,
  {\em On the stabilization of the trace formnula}, volume~1. International
  Press, 2011.

\end{thebibliography}
\bibliographystyle{alpha}

\end{document}